\documentclass[onefignum,onetabnum]{siamart220329}
\usepackage{amsfonts}
\usepackage{mathtools}
\usepackage{physics}
\usepackage[sort,compress]{cite}

\usepackage{subcaption}
\ifpdf 
\DeclareGraphicsExtensions{.eps,.pdf,.png,.jpg}
\else
\DeclareGraphicsExtensions{.eps}
\fi

\newsiamremark{remark}{Remark}
\newsiamthm{assumption}{Assumption}

\crefname{remark}{Remark}{Remarks}
\Crefname{remark}{Remark}{Remarks}
\crefname{assumption}{Assumption}{Assumptions}
\Crefname{assumption}{Assumptions}{Assumptions}
\Crefname{figure}{Fig.}{Figs.}
\crefname{figure}{Fig.}{Figs.}

\newcommand{\etal}{\emph{et al.}\ }
\newcommand{\T}{\mathrm{T}}
\newcommand{\offsetunderline}[1]{\underline{#1\mkern-2mu}\mkern2mu}
\renewcommand{\vec}[1]{\boldsymbol{#1}}
\newcommand{\vc}[1]{\underline{#1}}
\newcommand{\fn}[1]{\MakeUppercase{#1}}
\newcommand{\fvec}[1]{\vec{\fn{#1}}}
\newcommand{\mat}[1]{\offsetunderline{\offsetunderline{\mathit{#1}}}}
\newcommand{\set}[1]{\mathcal{\MakeUppercase{#1}}}
\newcommand{\opr}[1]{\mathcal{\MakeUppercase{#1}}}
\newcommand{\R}{\mathbb{R}} 
\newcommand{\N}{\mathbb{N}}
\renewcommand{\P}{\mathbb{P}}
\renewcommand{\S}{\mathbb{S}}
\newcommand{\nvolnodes}{N_q}
\newcommand{\nfacnodes}{N_{q_f}^{(\zeta)}}
\newcommand{\nelem}{N_e}
\newcommand{\nfac}{N_f}
\newcommand{\npoly}{{N_p^*}}

\headers{SBP Operators for Triangular/Tetrahedral SEM}{Tristan Montoya and David W. Zingg}

\title{Efficient Tensor-Product Spectral-Element Operators with the Summation-by-Parts Property on Curved Triangles and Tetrahedra\thanks{Some of the material presented in this article has appeared in: {\sc T.~Montoya and D.~W. Zingg}, {\em Stable and conservative high-order methods on triangular elements using tensor-product summation-by-parts operators}, Eleventh International Conference on Computational Fluid Dynamics, 2022.}\funding{This work was supported by the Natural Sciences and Engineering Research Council of Canada (NSERC), the Ontario Graduate Scholarship program, and the University of Toronto.}}

\author{Tristan Montoya\thanks{Institute for Aerospace Studies, University of Toronto, 4925 Dufferin St, Toronto, ON M3H 5T6, Canada
(\email{tristan.montoya@mail.utoronto.ca}, \email{david.zingg@utoronto.ca}).}
\and David W. Zingg\footnotemark[2]}

\ifpdf
\hypersetup{
pdftitle={Efficient Tensor-Product Spectral-Element Operators with the Summation-by-Parts Property on Curved Triangles and Tetrahedra},
pdfauthor={T. Montoya and D. W. Zingg}
}
\fi

\begin{document}
\maketitle

\begin{abstract}
We present an extension of the summation-by-parts (SBP) framework to tensor-product spectral-element operators in collapsed coordinates. The proposed approach enables the construction of provably stable discretizations of arbitrary order which combine the geometric flexibility of unstructured triangular and tetrahedral meshes with the efficiency of sum-factorization algorithms. Specifically, a methodology is developed for constructing triangular and tetrahedral spectral-element operators of any order which possess the SBP property (i.e.\ satisfying a discrete analogue of integration by parts) as well as a tensor-product decomposition. Such operators are then employed within the context of discontinuous spectral-element methods based on nodal expansions collocated at the tensor-product quadrature nodes as well as modal expansions employing Proriol-Koornwinder-Dubiner polynomials, the latter approach resolving the time step limitation associated with the singularity of the collapsed coordinate transformation. Energy-stable formulations for curvilinear meshes are obtained using a skew-symmetric splitting of the metric terms, and a weight-adjusted approximation is used to efficiently invert the curvilinear modal mass matrix. The proposed schemes are compared to those using non-tensorial multidimensional SBP operators, and are found to offer comparable accuracy to such schemes in the context of smooth linear advection problems on curved meshes, but at a reduced computational cost for higher polynomial degrees.
\end{abstract}



\section{Introduction}\label{sec:intro}
High-fidelity simulations of multiscale phenomena in science and engineering governed by time-dependent conservation laws require efficient, robust, and automated numerical methods suitable for complex geometries. Discontinuous spectral-element methods (DSEMs)\footnote{In this work, we use the term \emph{spectral-element method} (SEM) to refer to any approximation achieving high-order accuracy through the use of multiple degrees of freedom within a given element.} provide a promising approach to meeting these needs, as they are amenable to performant algorithms on modern hardware and are flexible in their support for local adaptation to geometric and solution features by varying the element size (i.e.\ $h$-adaptivity) or by varying the polynomial degree of the approximation within each element (i.e.\ $p$-adaptivity). While robustness has traditionally been a concern for such methods, modern formulations based on the summation-by-parts (SBP) property produce mathematical guarantees that discretizations will respect certain properties of the partial differential equations they approximate, facilitating the construction of provably stable and conservative high-order methods for a wide variety of problems. 
\par
The SBP approach, originally proposed by Kreiss and Scherer in the context of finite-difference methods \cite{kreiss_scherer_sbp_74}, has proven instrumental in bringing about a recent paradigm shift in the construction and analysis of linearly and nonlinearly stable DSEMs for conservation laws (see, for example, Chen and Shu \cite{chen_shu_dgsbp_review_19} and Gassner and Winters \cite{gassner_winters_novel_robust_dg_21} for reviews of such developments). Although the multidimensional SBP property introduced by Hicken \etal \cite{hicken_mdsbp_16} facilitates the construction of provably stable schemes suitable for complex geometries and solution adaptivity through the use of triangular and tetrahedral elements, the operators which constitute such discretizations lack a tensor-product structure. As a consequence, such schemes are not amenable to the \emph{sum-factorization} algorithms commonly employed on quadrilateral and hexahedral elements, wherein tensor-product operators are applied in a dimension-by-dimension manner, a highly efficient approach originating with the work of Orszag in the context of spectral methods \cite{orszag_spectral_complex_geometries_80}. Considering a polynomial spectral-element approximation of degree $p$ in $d$ spatial dimensions, local operations such as spatial differentiation are typically of $O(p^{2d})$ complexity when no such tensor-product structure is exploited. Hence, noting that sum factorization generally results in algorithms of $O(p^{d+1})$ complexity, the benefits of existing triangular and tetrahedral SBP operators with respect to geometric flexibility and suitability for adaptation are obtained at the expense of a significantly increased operation count relative to tensor-product operators on quadrilaterals and hexahedra, particularly at higher polynomial degrees.

\begin{figure}[!t]
\centering
\begin{subfigure}[c]{0.277\textwidth}
\includegraphics[width=0.84\textwidth]{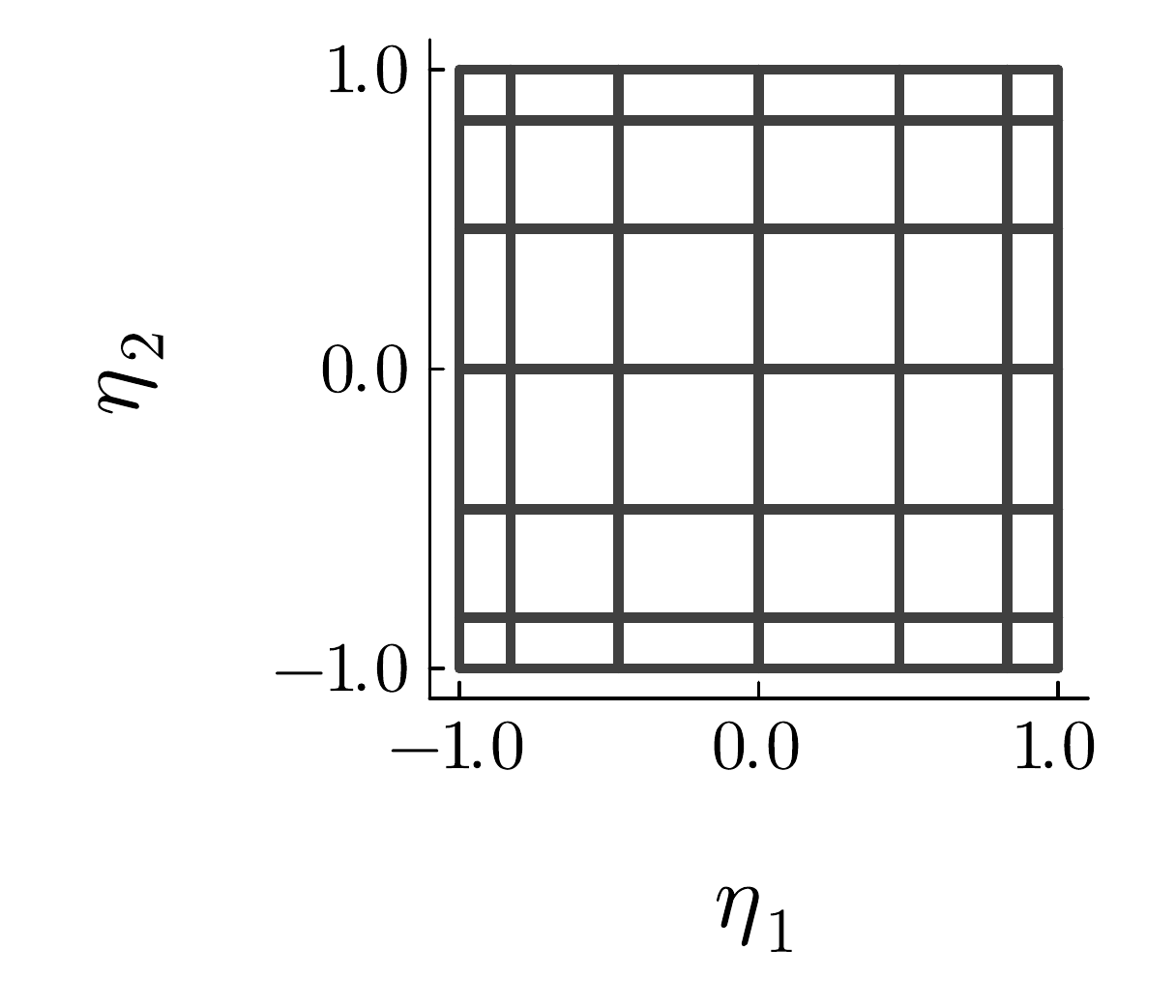}
\end{subfigure}~
\begin{subfigure}[c]{0.15\textwidth}
\centering
\Large $\longrightarrow$\\$\vec{\chi}$
\end{subfigure}~
\begin{subfigure}[c]{0.277\textwidth}
\includegraphics[width=0.84\textwidth]{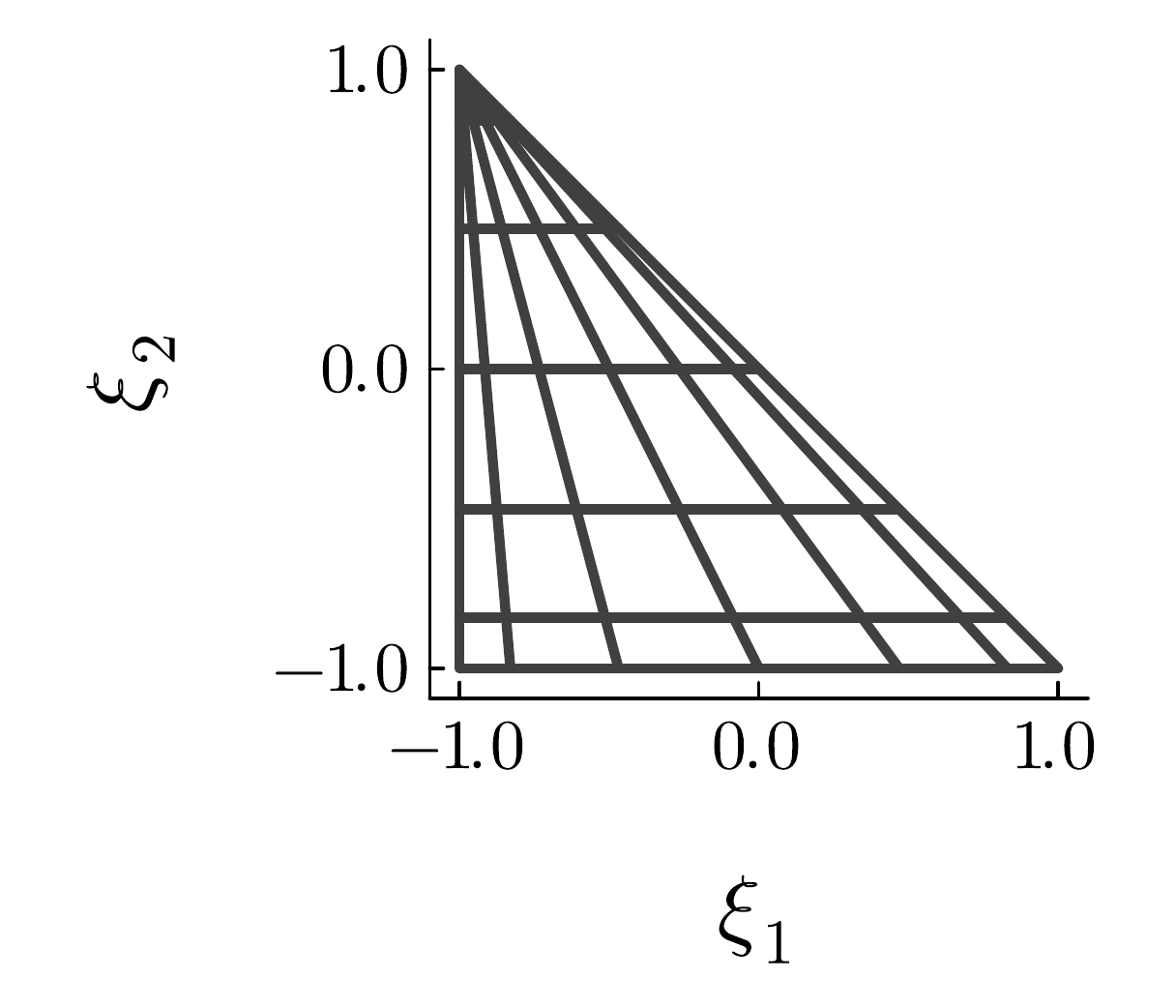}
\end{subfigure}\\
\begin{subfigure}[c]{0.277\textwidth}
\includegraphics[width=\textwidth]{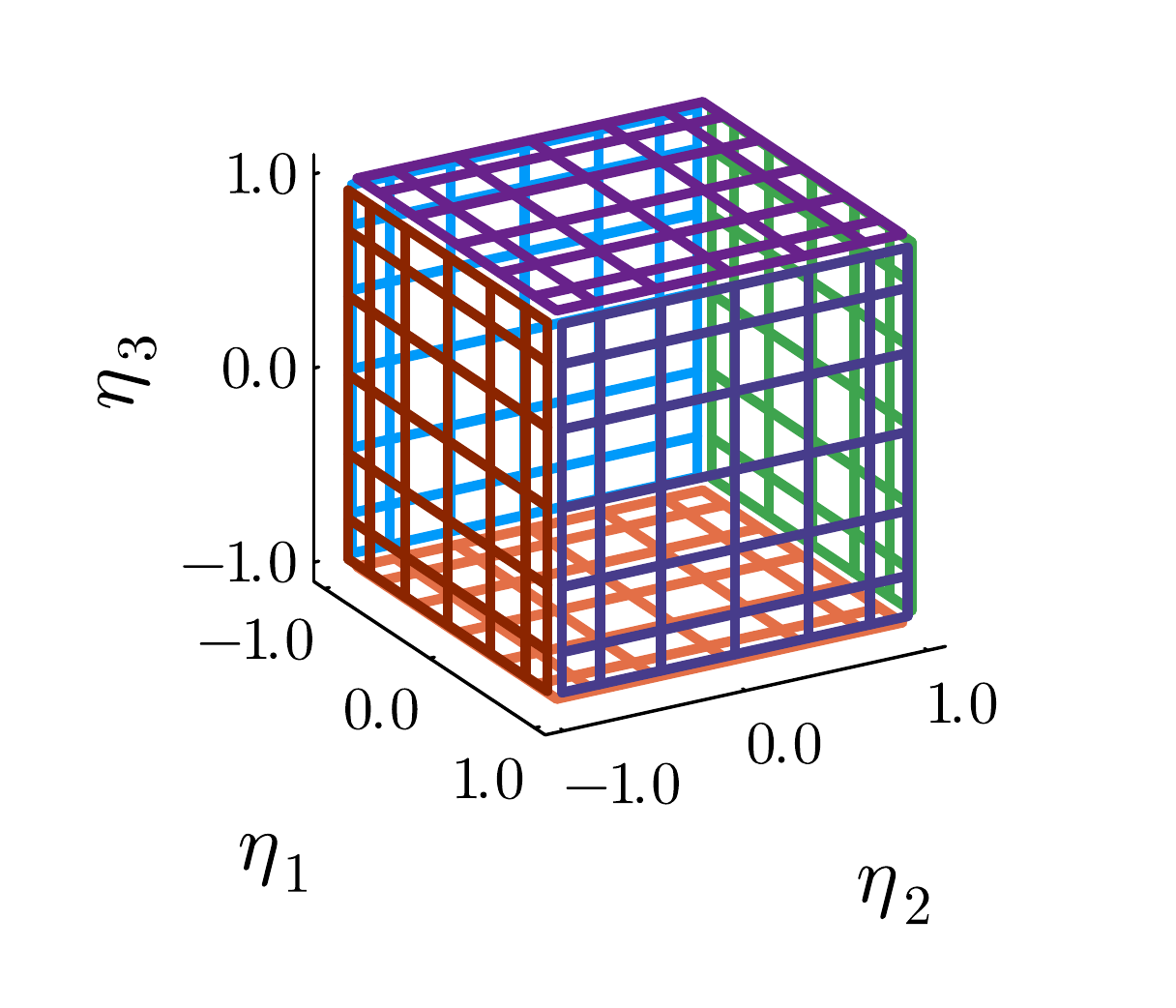}
\end{subfigure}~
\begin{subfigure}[c]{0.15\textwidth}
\centering
\Large $\longrightarrow$\\$\vec{\chi}$
\end{subfigure}~
\begin{subfigure}[c]{0.277\textwidth}
\includegraphics[width=\textwidth]{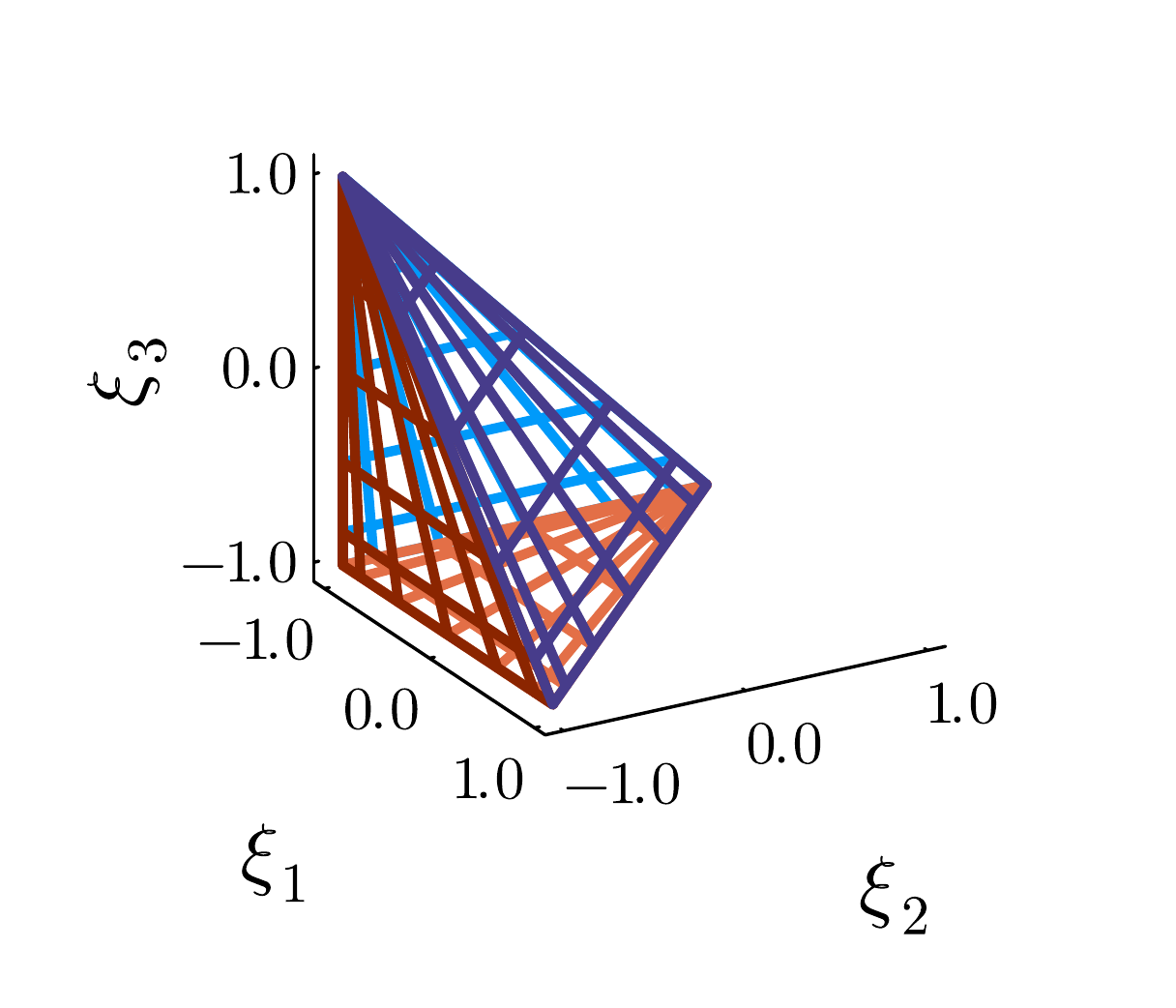}
\end{subfigure}
\caption{Illustration of the collapsed coordinate transformation $\vec{\xi} = \vec{\chi}(\vec{\eta})$ from the square to the reference triangle (top) and from the cube to the reference tetrahedron (bottom)}\label{fig:mapping}
\end{figure}

The use of a \emph{collapsed coordinate transformation} as shown in \cref{fig:mapping}, sometimes referred to as a \emph{Duffy transformation} \cite{duffy_collapsed_coordinates_82}, allows for the geometric flexibility of simplicial elements (as well as other element types such as prisms and pyramids) to be exploited alongside the aforementioned computational benefits of a tensor-product operator structure. Beginning with theoretical development by Dubiner \cite{dubiner_spectral_triangle_91} and the subsequent application of such ideas to continuous Galerkin (CG) methods \cite{sherwin_karniadakis_triangular_sem_95,sherwin_karniadakis_tetrahedra_hp_fem_96} as well as discontinuous Galerkin (DG) methods \cite{lomtev_karniadakis_dg_99,kirby_spectral_hp_dg_hybrid_grids_00}, collapsed-coordinate formulations have proven essential to the efficient use of triangular and tetrahedral SEMs at higher polynomial degrees. Moreover, such schemes have been shown in recent years to be amenable to efficient implementation on modern hardware (see, for example, Moxey \etal \cite{moxey_matrix_free_triangles_20}). Despite these successes, which are exemplified by the open-source spectral-element framework \emph{Nektar}\texttt{++} \cite{cantwell_nektar_2015}, existing tensor-product formulations in collapsed coordinates are subject to the aforementioned stability issues afflicting conventional SEM discretizations, which generally rely on \emph{ad hoc} techniques such as modal filtering (which requires careful tuning in order to avoid excessively contaminating the numerical solution \cite[section 5.3]{hesthaven08}) or over-integration (which incurs additional cost while still not ensuring stability \cite{winters_moura_mengaldo_overintegration_vs_splitform_18}) to achieve robustness in practice. Discretizations based on the SBP property, by contrast, do not require such problem-dependent remedies to ensure robustness, and may be systematically applied to complex problems with stability and conservation properties guaranteed by construction, without reliance on the exact integration of variational forms to establish such results.\footnote{This is important due to the aforementioned cost of over-integration as well as the fact that the variational formulations of many conservation laws of practical interest (for example, the conservative forms of the compressible Euler and Navier-Stokes equations) involve non-polynomial terms which are impractical, if not impossible, to integrate exactly using standard quadrature rules.}
\par
Motivated by the desire for efficient, provably stable discretizations of arbitrary order on triangles and tetrahedra, this paper presents a comprehensive extension of the SBP framework to tensor-product spectral-element operators in collapsed coordinates. Although the focus of this paper is on energy-stable discretizations for linear problems, the proposed operators also constitute the fundamental building blocks of entropy-stable schemes for nonlinear problems. Our main contributions are outlined below.
\begin{itemize}
\item A methodology is presented for constructing tensor-product spectral-element operators of any order with the SBP property on the reference triangle and tetrahedron, which are used within a split formulation to obtain conservative, free-stream preserving, and energy-stable DSEMs on curvilinear meshes. 
\item Nodal and modal variants of the above approach are proposed, which differ primarily due to the presence of an additional polynomial projection in the latter case, which alleviates the explicit time step restriction resulting from the singular nature of the collapsed coordinate transformation.
\item In the case of the modal formulation, we exploit the ``warped'' tensor-product structure of the Proriol-Koornwinder-Dubiner polynomial basis functions \cite{proriol_polynomials_57,koornwinder_orthogonal_polynomials_75,dubiner_spectral_triangle_91} alongside a weight-adjusted approximation to the inverse of the curvilinear mass matrix proposed by Chan \etal \cite{chan_weight_adjusted_dg_curvilinear_17}, resulting in an explicit algorithm for computing the time derivative on a given element in $O(p^{d+1})$ floating-point operations, with minimal storage requirements. 
\item Considering smooth linear advection problems on curved triangular and tetrahedral meshes, the proposed nodal and modal schemes both exhibit optimal $O(h^{p+1})$ convergence with respect to the element size $h$ when the mesh is refined with the degree $p$ held fixed and an upwind numerical flux is used. In the case of the modal formulation with an upwind or central numerical flux, exponential convergence is observed when $p$ is increased with $h$ held fixed.
\item Compared to methods of the same degree employing symmetric nodal sets on triangles and tetrahedra, the proposed schemes are found to be similar in accuracy and (in the case of the modal approach) spectral radius, but require significantly fewer floating-point operations at higher polynomial degrees.
\end{itemize}
We now describe the structure of the remainder of this paper. In \cref{sec:preliminaries}, we review the SBP framework and associated notation employed throughout this work. Tensor-product SBP operators on the reference triangle and tetrahedron are then introduced in \cref{sec:tri,sec:tet}, respectively, and are used to construct nodal and modal DSEM formulations on curvilinear meshes in \cref{sec:dsem}. We describe strategies for efficient implementation of the proposed schemes in \cref{sec:implementation}. Numerical experiments are presented and discussed in \cref{sec:numerical}, and concluding remarks are provided in \cref{sec:conclusions}.

\section{Preliminaries}\label{sec:preliminaries}
\subsection{Notation} 
The notation in this paper follows that introduced by the authors in \cite{montoya_tensor_product_22} and \cite{montoya_unifying_21}. Single underlines are used to denote vectors (treated as column matrices), whereas double underlines denote matrices. Symbols in bold are used specifically to denote Cartesian (i.e.\ spatial) vectors, for which we employ the usual dot product $\vec{x} \cdot \vec{y} \coloneqq x_1y_1 + \cdots + x_dy_d$ and Euclidean norm $\lVert \vec{x} \rVert^2 \coloneqq \vec{x} \cdot \vec{x}$, and del operator $\nabla_{\vec{x}}\coloneqq [\partial/\partial x_1 ,\ldots, \partial/\partial x_d]^\T$. The symbols $\R$, $\R^+$, $\R_0^+$, $\N$, $\N_0$, and $\S^{d-1}$ denote the real numbers, the positive real numbers, the non-negative real numbers, the natural numbers (excluding zero), the natural numbers including zero, and the unit $(d-1)$-sphere $\S^{d-1} \coloneqq \{ \vec{x} \in \R^d : \lVert \vec{x} \rVert = 1\}$, respectively. The symbols $\vc{0}^{(N)}$ and $\vc{1}^{(N)}$ are reserved for vectors of length $N \in \N$ containing all zeros and all ones, respectively, and we use $\{1:N\}$ as shorthand for $\{1,2,\ldots,N\}$. Given a bounded domain $\set{D} \subset \R^d$, we use $\partial\set{D}$ to denote its boundary and $\bar{\set{D}} \coloneqq \set{D} \cup \partial \set{D}$ to denote its closure; the interior of a closed domain $\set{D}$ is then given by $\mathring{\set{D}} \coloneqq \set{D} \setminus \partial\set{D}$. The space of polynomials of maximum total degree $p \in \N_0$ on $\set{D}$ is then defined as $\P_p(\set{D}) \coloneqq \operatorname{span}\{\set{D} \ni \vec{x} \mapsto x_1^{\alpha_1} \cdots x_d^{\alpha_d} : \vec{\alpha} \in \set{N}(p)\}$ in terms of the multi-index set $\set{N}(p) \coloneqq \{\vec{\alpha} \in \N_0^d : \alpha_1 + \cdots + \alpha_d \leq p\}$, which is of cardinality $
\npoly \coloneqq \binom{p+d}{d}$. Other relevant notational conventions and definitions are introduced as they appear. 

\subsection{Summation-by-parts operators}
Whether or not explicit in their construction, existing energy-stable and entropy-stable spectral-element methods on simplicial elements typically rely on the multidimensional summation-by-parts property introduced by Hicken \etal \cite{hicken_mdsbp_16}, who proposed a definition equivalent to the following.
\begin{definition}[Nodal SBP operator]\label{def:sbp}
Let $\hat{\Omega} \subset \R^d$ denote a compact, connected domain, on which we define a set of $\nvolnodes \in \N$ distinct nodes $\{\vec{\xi}^{(i)}\}_{i\in \{1:\nvolnodes\}} \subset \hat{\Omega}$, and let $\vc{u} \coloneqq [\fn{u}(\vec{\xi}^{(1)}),\ldots, \fn{u}(\vec{\xi}^{(\nvolnodes)})]^\T$ and $\vc{v} \coloneqq [\fn{v}(\vec{\xi}^{(1)}),\ldots, \fn{v}(\vec{\xi}^{(\nvolnodes)})]^\T$ contain the nodal values of functions $\fn{u},\fn{v} : \hat{\Omega} \to \R$. A matrix $\mat{D}^{(m)} \in \R^{\nvolnodes \times \nvolnodes}$ approximating the partial derivative $\partial/\partial \xi_m$ is then a \emph{nodal SBP operator} of degree $p \in \N_0$ if it satisfies
\begin{equation}\label{eq:accuracy}
\mat{D}^{(m)}\vc{v} = \big[(\partial\fn{v}/\partial\xi_m)(\vec{\xi}^{(1)}), \ldots, (\partial\fn{v}/\partial\xi_m)(\vec{\xi}^{(\nvolnodes)}) \big]^\T, \quad \forall \, \fn{v} \in \P_p(\hat{\Omega}),
\end{equation}
and may be decomposed as $\mat{D}^{(m)} = \mat{W}^{-1}\mat{Q}^{(m)}$ such that $\mat{W}\in \R^{\nvolnodes\times \nvolnodes}$ is symmetric positive-definite (SPD), $\mat{Q}^{(m)}\in \R^{\nvolnodes\times \nvolnodes}$ satisfies $
\mat{Q}^{(m)} + (\mat{Q}^{(m)})^\T = \mat{E}^{(m)}$, and 
\begin{equation}\label{eq:facet_accuracy}
\vc{u}^\T\mat{E}^{(m)}\vc{v} = \int_{\partial\hat{\Omega}} \fn{u}(\vec{\xi})\fn{v}(\vec{\xi}) \hat{n}_m(\vec{\xi}) \, \dd\hat{s}, \quad \forall \, \fn{u},\fn{v} \in \P_r(\hat{\Omega}),
\end{equation}
holds for some $r \geq p$, where $\hat{\vec{n}} : \partial\hat{\Omega} \to \S^{d-1}$ is the outward unit normal vector to $\hat{\Omega}$.
\end{definition}
The SBP property serves as a discrete analogue of integration by parts, as given by
\begin{equation}
\begin{alignedat}{3}
\underset{\rotatebox{90}{$\, \approx $}}{\int_{\hat{\Omega}} \fn{u}(\vec{\xi}) \pdv{\fn{v}(\vec{\xi})}{\xi_m} \, \dd \vec{\xi} }  &+ \underset{\rotatebox{90}{$\, \approx $}}{\int_{\hat{\Omega}}\pdv{\fn{u}(\vec{\xi})}{\xi_m} \fn{v}(\vec{\xi})  \, \dd \vec{\xi}} & = & \  \underset{\rotatebox{90}{$\, \approx $}}{ \int_{\partial\hat{\Omega}} \fn{u}(\vec{\xi})\fn{v}(\vec{\xi}) \hat{n}_m(\vec{\xi}) \, \dd\hat{s}}. \\ 
\vc{u}^\T\mat{Q}^{(m)}\vc{v} \qquad &+ \quad \vc{u}^\T\big(\mat{Q}^{(m)}\big)^\T\vc{v} & \ = & \qquad\quad \vc{u}^\T\mat{E}^{(m)}\vc{v}
\end{alignedat}
\end{equation}
We refer to any SBP operator for which the associated matrix $\mat{W}$ is diagonal as a \emph{diagonal-norm} SBP operator, referring to the role of such a matrix in defining a discrete norm in which stability may be proven. In such a case, the diagonal entries of $\mat{W}$ constitute the weights $\{\omega^{(i)}\}_{i \in \{1:\nvolnodes\}} \subset \R^+$ for a quadrature rule satisfying
\begin{equation}\label{eq:sbp_quadrature}
\sum_{i=1}^{\nvolnodes} \fn{v}(\vec{\xi}^{(i)})\, \omega^{(i)} = \int_{\hat{\Omega}} \fn{v}(\vec{\xi})\, \dd \vec{\xi}, \quad \forall \, \fn{v} \in \P_\tau(\hat{\Omega}),
\end{equation}
where $\tau \geq 2p - 1$ for any diagonal-norm SBP operator of degree $p$ \cite[Theorem 3.2]{hicken_mdsbp_16}. 

\subsection{Decomposition of the boundary operators}\label{sec:sbp_boundary}
To employ SBP operators satisfying the conditions of \cref{def:sbp} within the context of an SEM, we construct such operators on a canonical reference element and use a coordinate transformation to obtain operators on each physical element of the mesh. Let us now assume that the reference element $\hat{\Omega} \subset \R^d$ is a polytope and partition its boundary into $\nfac \in \N$ closed subsets $\{ \hat{\Gamma}^{(\zeta)}\}_{\zeta \in \{1:\nfac\}}$ with disjoint interiors, referred to as \emph{facets}, on which the outward unit normal $\vec{n}^{(\zeta)} \in \S^{d-1}$ is assumed to be constant. On each facet, we introduce $\nfacnodes \in \N$  quadrature nodes and corresponding weights, which are given by
\begin{equation}
\{\vec{\xi}^{(\zeta,i)}\}_{i \in \{1:\nfacnodes\}} \subset \hat{\Gamma}^{(\zeta)}, \quad  \{\omega^{(\zeta,i)}\}_{i \in \{1:\nfacnodes\}} \subset \R_0^+.
\end{equation}
As in Del Rey Fern\'andez \etal \cite[section 3]{delrey_mdsbp_sat_18}, we then restrict our attention to the class of SBP operators for which the boundary matrices take the form
\begin{equation}\label{eq:e_decomp}
\mat{E}^{(m)} \coloneqq \sum_{\zeta=1}^{\nfac} \hat{n}_m^{(\zeta)} \big(\mat{R}^{(\zeta)}\big)^\T \mat{B}^{(\zeta)}\mat{R}^{(\zeta)},
\end{equation}
where $\mat{B}^{(\zeta)} \in \R^{\nfacnodes\times\nfacnodes}$ is a diagonal matrix with entries $B_{ij}^{(\zeta)} \coloneqq \omega^{(\zeta,i)}\delta_{ij}$, and $\mat{R}^{(\zeta)} \in \R^{\nfacnodes \times \nvolnodes}$ is an interpolation/extrapolation operator of degree $r^{(\zeta)} \geq p$, satisfying
\begin{equation}\label{eq:extrap_acc}
\mat{R}^{(\zeta)}\vc{v} = \big[\fn{v}(\vec{\xi}^{(\zeta,1)}), \ldots, \fn{v}(\vec{\xi}^{(\zeta,\nfacnodes)})\big]^\T, \quad \forall \, \fn{v} \in \P_{r^{(\zeta)}}(\hat{\Omega}).
\end{equation}
Such a decomposition is critical for the weak imposition of boundary conditions and interface coupling for discontinuous approximation spaces, as in the case of a DSEM.

\subsection{Curvilinear meshes}\label{sec:meshmap}
Let $\Omega \subset \R^d$ denote an open, bounded, connected spatial domain with a piecewise smooth boundary $\partial\Omega$ and outward unit normal vector $\vec{n} : \partial\Omega \to \S^{d-1}$. The domain is partitioned into a mesh $\{\Omega^{(\kappa)}\}_{\kappa \in \{1:\nelem\}}$ consisting of $\nelem \in \N$ compact, connected elements of characteristic size $h \in \R^+$, satisfying
\begin{equation}
\bigcup_{\kappa=1}^{\nelem} \Omega^{(\kappa)} = \bar{\Omega} \quad \text{and} \quad \mathring{\Omega}^{(\kappa)} \cap \mathring{\Omega}^{(\nu)} = \emptyset, \quad \forall \, \kappa \neq \nu.
\end{equation}
Each element is further assumed to be the image of a polytopal reference element $\hat{\Omega} \subset \R^d$ under a smooth, time-invariant mapping $\fvec{x}^{(\kappa)} :\hat{\Omega} \rightarrow \Omega^{(\kappa)}$. The Jacobian of such a mapping is denoted by $\nabla_{\vec{\xi}}\fvec{x}^{(\kappa)}(\vec{\xi}) \in \R^{d\times d}$, where the determinant $\fn{J}^{(\kappa)}(\vec{\xi}) \coloneqq \operatorname{det}(\nabla_{\vec{\xi}}\fvec{x}^{(\kappa)}(\vec{\xi}))$ is assumed to be positive for all $\vec{\xi} \in \hat{\Omega}$. The outward unit normal vector $\vec{n}^{(\kappa,\zeta)} : \Gamma^{(\kappa,\zeta)} \to \S^{d-1}$ to the facet $\Gamma^{(\kappa,\zeta)} \subset \partial\Omega^{(\kappa)}$, which is the image of $\hat{\Gamma}^{(\zeta)} \subset \partial\hat{\Omega}$ under the mapping $\fvec{x}^{(\kappa)}$, is then given according to Nanson's formula,
\begin{equation}\label{eq:nanson}
\fn{J}^{(\kappa,\zeta)}(\vec{\xi})\vec{n}^{(\kappa,\zeta)}(\fvec{x}^{(\kappa)}(\vec{\xi})) = \fn{J}^{(\kappa)}(\vec{\xi})(\nabla_{\vec{\xi}}\fvec{x}^{(\kappa)}(\vec{\xi}))^{-\T}\vec{n}^{(\zeta)},
\end{equation}
where we define $\fn{J}^{(\kappa,\zeta)}(\vec{\xi}) \coloneqq \lVert\fn{J}^{(\kappa)}(\vec{\xi})(\nabla_{\vec{\xi}}\fvec{x}^{(\kappa)}(\vec{\xi}))^{-\T}\rVert$. We then form the diagonal matrices $\mat{J}^{(\kappa)}, \mat{\mathit{\Lambda}}^{(\kappa,l,m)} \in \R^{\nvolnodes\times\nvolnodes}$ and $\mat{J}^{(\kappa,\zeta)}, \mat{N}^{(\kappa,\zeta,m)} \in \R^{\nfacnodes \times \nfacnodes}$ given by
\begin{equation}\label{eq:geometric_factors}
\begin{alignedat}{2}
J_{ij}^{(\kappa)} &\coloneqq \fn{J}^{(\kappa)}(\vec{\xi}^{(i)})\delta_{ij}, \quad &\mathit{\Lambda}_{ij}^{(\kappa,l,m)} &\coloneqq [\fn{J}^{(\kappa)}(\vec{\xi}^{(i)})(\nabla_{\vec{\xi}}\fvec{x}^{(\kappa)}(\vec{\xi}^{(i)}))^{-1}]_{lm} \delta_{ij},\\
J_{ij}^{(\kappa,\zeta)} &\coloneqq \fn{J}^{(\kappa,\zeta)}(\vec{\xi}^{(\zeta,i)})\delta_{ij}, \quad & N_{ij}^{(\kappa,\zeta,m)} &\coloneqq n_m^{(\kappa,\zeta)}(\fvec{x}^{(\kappa)}(\vec{\xi}^{(\zeta,i)}))\delta_{ij},
\end{alignedat}
\end{equation}
containing the values of the geometric factors and normals at the quadrature nodes. 

\subsection{Approximation on the physical element}\label{sec:sbp_phys}
Given any diagonal-norm SBP operator on the reference element, we can use the split (i.e.\ skew-symmetric) formulation proposed by Crean \etal \cite[section 5]{crean_entropystable_sbp_euler_curved_17} to approximate $\partial / \partial x_m$ under the mapping $\vec{x} = \fvec{x}^{(\kappa)}(\vec{\xi})$ using the operator $\mat{D}^{(\kappa,m)} \coloneqq (\mat{W}\mat{J}^{(\kappa)})^{-1}\mat{Q}^{(\kappa,m)}$, where we define
\begin{equation}\label{eq:q_decomp}
\mat{Q}^{(\kappa,m)} \coloneqq  \frac{1}{2}\sum_{l=1}^d\Big(\mat{\Lambda}^{(\kappa,l,m)}\mat{W}\mat{D}^{(l)} - \big(\mat{D}^{(l)}\big)^\T\mat{W}\mat{\Lambda}^{(\kappa,l,m)}\Big) + \frac{1}{2}\mat{E}^{(\kappa,m)},
\end{equation}
with
\begin{equation}\label{eq:e_phys}
\mat{E}^{(\kappa,m)} \coloneqq \sum_{\zeta=1}^{\nfac}\big(\mat{R}^{(\zeta)}\big)^\T\mat{B}^{(\zeta)}\mat{J}^{(\kappa,\zeta)}\mat{N}^{(\kappa,\zeta,m)}\mat{R}^{(\zeta)}.
\end{equation}
While polynomials of degree $p$ on the physical element are not, in general, differentiated exactly when $\fvec{x}^{(\kappa)}$ is not an affine mapping, it was shown in \cite[Theorem 9]{crean_entropystable_sbp_euler_curved_17} that the error in the resulting split-form approximation to $\partial/\partial x_m$ is  of $O(h^p)$ under suitable assumptions. In addition to the SBP property on the physical element,
\begin{equation}
\mat{Q}^{(\kappa,m)} + \big(\mat{Q}^{(\kappa,m)}\big)^\T = \mat{E}^{(\kappa,m)},\quad \forall \, m \in \{1:d\},
\end{equation}
which follows directly from \cref{eq:q_decomp}, proofs of conservation, free-stream preservation, and nonlinear stability (i.e.\ entropy stability) typically require the \emph{discrete metric identities} to be satisfied, as given by
\begin{equation}\label{eq:discrete_metric_identities}
\mat{D}^{(\kappa,m)}\vc{1}^{(\nvolnodes)} = \vc{0}^{(\nvolnodes)}, \quad \forall \, m \in \{1:d\}.
\end{equation}
Assuming that the metric terms and normals in \cref{eq:geometric_factors} are computed analytically, it was shown in \cite[Theorems 6 and 8]{crean_entropystable_sbp_euler_curved_17} that \cref{eq:discrete_metric_identities} holds for physical operators constructed as in \cref{eq:q_decomp} if the mapping is a polynomial $\fvec{x}^{(\kappa)} \in [\P_{p_g}(\hat{\Omega})]^d$ of degree $p_g \leq p + 1$ in two dimensions or $p_g \leq \lfloor p/2 \rfloor + 1$ in three dimensions. Otherwise, approximations such as the conservative curl form proposed by Kopriva \cite[section 7]{kopriva_metric_identities_discontinuous_sem_curved_06} and extended to simplicial meshes by Chan and Wilcox \cite[section 5]{chan_wilcox_entropystable_curvilinear_19} or the optimization-based approach introduced in \cite[section 5.4]{crean_entropystable_sbp_euler_curved_17} may be used to ensure that \cref{eq:discrete_metric_identities} is satisfied.

\section{Tensor-product approximations on the reference triangle}\label{sec:tri}
We now turn to the main theoretical contribution of this paper, which is to construct a new class of SBP operators of arbitrary order employing tensor-product approximations in collapsed coordinates. We will begin with the triangular case in this section, describing the extension to tetrahedra in \cref{sec:tet}. The reference domain is taken here to be
\begin{equation}\label{eq:reference_triangle}
\hat{\Omega} \coloneqq \big\{\vec{\xi} \in [-1,1]^2 : \xi_1 + \xi_2 \leq 0  \big\},
\end{equation}
where, as a convention, we number the facets (i.e.\ edges of the triangle) as
\begin{equation}\label{eq:tri_facets}
\begin{gathered}
\hat{\Gamma}^{(1)} \coloneqq \big\{ \vec{\xi} \in \hat{\Omega} : \xi_2 = -1 \big\}, \quad \hat{\Gamma}^{(2)} \coloneqq \big\{ \vec{\xi} \in \hat{\Omega} : \xi_1 + \xi_2 = 0 \big\}, \\
\hat{\Gamma}^{(3)} \coloneqq \big\{ \vec{\xi} \in \hat{\Omega} : \xi_1 = -1 \big\}.
\end{gathered}
\end{equation}
The outward unit normal vectors to each facet are then given by
\begin{equation}\label{eq:normals_tri}
\hat{\vec{n}}^{(1)} = \mqty[0\\-1], \quad \hat{\vec{n}}^{(2)} = \mqty[1/\sqrt{2}\\ 1/\sqrt{2}], \quad \hat{\vec{n}}^{(3)} = \mqty[-1\\0].
\end{equation}
The proposed schemes are constructed based on a \emph{collapsed coordinate system} $\vec{\eta} \in [-1,1]^2$, from which any point can be mapped onto the \emph{reference coordinate system} $\vec{\xi} \in \hat{\Omega}$ through the mapping $\vec{\chi} : [-1,1]^2 \to \hat{\Omega}$ depicted in \cref{fig:mapping}, which is given by
\begin{equation}\label{eq:tri_mapping}
\vec{\chi}(\vec{\eta}) \coloneqq  \mqty[\tfrac{1}{2}(1+\eta_1)(1-\eta_2) - 1\\ \eta_2],
\end{equation}
where we adopt a similar notation to that employed by Karniadakis and Sherwin \cite[section 3.2]{karniadakis_sherwin_spectral_hp_element}. Such a mapping has an inverse defined for $\vec{\xi} \in \hat{\Omega} \setminus \{[-1,1]^\T\}$ by
\begin{equation}\label{eq:inverse_tri}
\vec{\chi}^{-1}(\vec{\xi}) \coloneqq  \mqty[ 2(1+\xi_1)/(1-\xi_2) - 1\\\xi_2].
\end{equation}
It is then straightforward to show that integrals on the reference element can be expressed in terms of the collapsed coordinate system as
\begin{equation}\label{eq:integral_tri}
\int_{\hat{\Omega}} \fn{v}(\vec{\xi}) \, \dd \vec{\xi} = \int_{-1}^1\int_{-1}^1 \fn{v}(\vec{\chi}(\vec{\eta})) \frac{1-\eta_2}{2}  \,\dd \eta_1\dd \eta_2.
\end{equation}
Similarly, partial derivatives with respect to each reference coordinate may be computed in terms of the collapsed coordinate system via the chain rule as
\begin{subequations}\label{eq:chain_rule_tri}
\allowdisplaybreaks
\begin{align}
\frac{\partial \fn{v}}{\partial \xi_1} (\vec{\chi}(\vec{\eta})) &= \frac{2}{1-\eta_2} \frac{\partial}{\partial \eta_1} \fn{v}(\vec{\chi}(\vec{\eta})), \\
\frac{\partial \fn{v}}{\partial \xi_2} (\vec{\chi}(\vec{\eta})) &= \bigg(\frac{1+\eta_1}{1-\eta_2} \frac{\partial}{\partial \eta_1} +   \frac{\partial}{\partial \eta_2}\bigg) \fn{v}(\vec{\chi}(\vec{\eta})),
\end{align}
\end{subequations}
where we note that such expressions are undefined at the top edge (i.e.\ $\eta_2 = 1$), which collapses onto the singular vertex $\vec{\xi} = [-1,1]^\T$ under the transformation in \cref{eq:tri_mapping}.

\subsection{Nodal sets and interpolants}\label{sec:interp_tri} Let $q_1,q_2 \in \N_0$ denote the degrees of the polynomial approximations with respect to $\eta_1$ and $\eta_2$, respectively, 
and define $q \coloneqq \min(q_1,q_2)$. Note that we allow for the use of different polynomial degrees in each direction, where such flexibility could be exploited, for example, within the context of anisotropic $p$-adaptivity. Considering Jacobi weight functions of the form $\varpi^{(a,b)}(\eta) \coloneqq (1-\eta)^{a} (1+\eta)^{b}$ for generality, we construct Gaussian quadrature rules with nodes and weights
\begin{equation}\label{eq:quadrature_2d}
\begin{alignedat}{4}
\{\eta_1^{(i)}\}_{i \in \{0:q_1\}} &\subset [-1,1], \quad &&\{\eta_2^{(i)}\}_{i \in \{0:q_2\}} &&\subset [-1,1), \\
\{\omega_1^{(i)}\}_{i\in\{0:q_1\}} &\subset \R^+, \quad && \{\omega_2^{(i)}\}_{i\in\{0:q_2\}} &&\subset \R^+,
\end{alignedat}
\end{equation}
where the rule in the $\eta_m$ direction employs $q_m + 1$ distinct nodes and satisfies
\begin{equation}\label{eq:quad_1d}
\sum_{i=0}^{q_m} \fn{v}(\eta_m^{(i)})\, \omega_m^{(i)} = \int_{-1}^1 \fn{v}(\eta_m) \varpi^{(a_m,b_m)}(\eta_m)\, \dd \eta_m, \quad \forall \, \fn{v} \in \P_{\tau_m^{(a_m,b_m)}}([-1,1]),
\end{equation} 
such that $a_m,b_m > -1$ and $\tau_m^{(a_m,b_m)} = 2q_m + \delta$, where $\delta = 1$ for Jacobi-Gauss (JG) quadrature rules, $\delta = 0$ for Jacobi-Gauss-Radau (JGR) quadrature rules, and $\delta = -1$ for Jacobi-Gauss-Lobatto (JGL) quadrature rules.\footnote{Gaussian quadrature is discussed in more detail, for example, in \cite[appendix B]{karniadakis_sherwin_spectral_hp_element}. In the case of $(a_m,b_m) = (0,0)$, we recover the familiar Legendre-Gauss (LG), Legendre-Gauss-Radau (LGR), and Legendre-Gauss-Lobatto (LGL) quadrature rules for integration with respect to the unit weight.} Lagrange bases $\{\ell_m^{(i)}\}_{i \in \{0:q_m\}}$ are then defined on the above quadrature nodes for $m \in \{1:2\}$ as
\begin{equation}\label{eq:lagrange}
\ell_m^{(i)}(\eta_m) \coloneqq \prod_{j \in \{0:q_m\} \setminus \{i\}} \frac{\eta_m - \eta_m^{(j)}}{\eta_m^{(j)}-\eta_m^{(i)}},
\end{equation}
which can be used to construct the interpolant of a given function $\fn{v} : \hat{\Omega} \to \R$ as
\begin{equation}\label{eq:nodal_tensor}
(\opr{I}_q \fn{v})(\vec{\chi}(\vec{\eta})) \coloneqq \sum_{\alpha_1=0}^{q_1}\sum_{\alpha_2=0}^{q_2} \fn{v}(\vec{\chi}(\eta_1^{(\alpha_1)},\eta_2^{(\alpha_2)}))\ell_1^{(\alpha_1)}(\eta_1)\ell_2^{(\alpha_2)}(\eta_2).
\end{equation}
The above interpolant is, in general, a rational function with a singularity at $\vec{\xi} = [-1,1]^\T$ when expressed in reference coordinates under the inverse mapping in \cref{eq:inverse_tri}. However, despite being based on a rational function space (the approximation properties of which are analyzed by Shen \etal \cite{shen_fully_tensorial_triangular_sem_09} and Li and Wang \cite{li_rational_tetrahedra_10}), the interpolation is exact for all polynomials of up to degree $q$, as characterized by the following lemma.

\begin{lemma}\label{lem:accuracy_2d}
The tensor-product interpolation operator defined in \cref{eq:nodal_tensor} is exact for any polynomial $\fn{v} \in \P_q(\hat{\Omega})$, satisfying $(\opr{I}_q \fn{v})(\vec{\xi}) = \fn{v}(\vec{\xi})$ for all $\vec{\xi} \in \hat{\Omega} \setminus \{[-1,1]^\T\}$.
\end{lemma}
\begin{proof}
First, we note that any monomial of the form $\pi^{(\alpha_1,\alpha_2)}(\vec{\xi}) \coloneqq \xi_1^{\alpha_1}\xi_2^{\alpha_2}$ with $\vec{\alpha} \in \set{N}(q)$ can be expressed in terms of the collapsed coordinate system as
\begin{equation}\label{eq:monomial}
\begin{aligned}
\pi^{(\alpha_1,\alpha_2)}(\vec{\chi}(\vec{\eta})) &= \big(\tfrac{1}{2}(1+\eta_1)(1-\eta_2)-1\big)^{\alpha_1} \eta_2^{\alpha_2},\\
\end{aligned}
\end{equation}
which is of maximum degree $\alpha_1$ in $\eta_1$ and of maximum degree $\alpha_1 + \alpha_2$ in $\eta_2$. Since $\alpha_1 \leq q_1$ and $\alpha_1 + \alpha_2 \leq q_2$ for all $\vec{\alpha} \in \set{N}(q)$, it follows from expressing any polynomial $\fn{v} \in \P_q(\hat{\Omega})$ as a linear combination of monomials in the form of \cref{eq:monomial} that $\fn{v} \circ \vec{\chi}$ lies within the span of the tensor-product basis, and is therefore interpolated exactly.
\end{proof} 

\subsection{Volume quadrature}\label{sec:vol_quad_tri}
Let $\sigma : \{0:q_1\} \times \{0:q_2\} \to \{1:\nvolnodes\}$ represent a bijective mapping which defines an ordering of the $\nvolnodes \coloneqq (q_1 + 1)(q_2 + 1)$ volume quadrature nodes. Denoting by $\vec{\alpha} \in \{0:q_1\} \times \{0:q_2\}$ the multi-index associated with a given node, we therefore obtain the multidimensional quadrature nodes $\{\vec{\xi}^{(i)}\}_{i\in \{1:\nvolnodes\}} \subset \hat{\Omega}$ as $\vec{\xi}^{(\sigma(\vec{\alpha}))} \coloneqq \vec{\chi}(\eta_1^{(\alpha_1)},\eta_2^{(\alpha_2)})$, with corresponding weights $\{\omega^{(i)}\}_{i\in \{1:\nvolnodes\}} \subset \R^+$ given by
\begin{equation}\label{eq:legendre_weights_tri}
\omega^{(\sigma(\vec{\alpha}))} \coloneqq \frac{1-\eta_2^{(\alpha_2)}}{2}\omega_1^{(\alpha_1)}\omega_2^{(\alpha_2)}
\end{equation}
for $(a_1,b_1) = (0,0)$ and $(a_2,b_2) = (0,0)$, or, alternatively, by
\begin{equation}\label{eq:jacobi_weights_tri}
\omega^{(\sigma(\vec{\alpha}))} \coloneqq \frac{1}{2}\omega_1^{(\alpha_1)}\omega_2^{(\alpha_2)}
\end{equation}
if a Jacobi-type quadrature rule with $(a_2,b_2) = (1,0)$ is instead used as in \cite[section 2.2.1]{sherwin_karniadakis_triangular_sem_95} to subsume the factor of $1-\eta_2^{(\alpha_2)}$ in \eqref{eq:legendre_weights_tri} arising from the change of variables.
\subsection{Facet quadrature}\label{sec:fac_quad_tri}
To integrate over each facet, we let $q_f \in \N_0$ and introduce an additional set of nodes $\{\eta_f^{(i)}\}_{i \in \{0:q_f\}} \subset [-1,1]$ with corresponding quadrature weights $\{\omega_f^{(i)}\}_{i \in \{0:q_f\}}\subset \R_0^+$, defining a rule of degree $\tau_f^{(0,0)} \in \N_0$, satisfying
\begin{equation}\label{eq:tri_facet_quad}
\sum_{i=0}^{q_f}\fn{v}(\eta_f^{(i)})\,\omega_f^{(i)}  = \int_{-1}^1 \fn{v}(\eta_f) \, \dd \eta_f, \quad \forall \,\fn{V} \in \P_{\tau_f^{(0,0)}}([-1,1]).
\end{equation}
Using the numbering convention in \cref{eq:tri_facets}, we can then define the $\nfacnodes \coloneqq q_f + 1$ facet quadrature nodes $\{\vec{\xi}^{(\zeta,i)}\}_{i \in \{1:\nfacnodes\}} \subset \hat{\Gamma}^{(\zeta)}$ and weights $\{\omega^{(\zeta,i)}\}_{i \in \{1:\nfacnodes\}} \subset \R_0^+$ as
\begin{equation}\label{eq:tri_facet_quadrature_nodes}
\begin{alignedat}{3}
\vec{\xi}^{(1,i)} &\coloneqq \vec{\chi}(\eta_f^{(i-1)},-1), \quad &\vec{\xi}^{(2,i)} &\coloneqq \vec{\chi}(1,\eta_f^{(i-1)}), \quad &\vec{\xi}^{(3,i)} &\coloneqq \vec{\chi}(-1,\eta_f^{(i-1)}),\\
\omega^{(1,i)} &\coloneqq \omega_f^{(i-1)}, \quad &\omega^{(2,i)} &\coloneqq \sqrt{2}\omega_f^{(i-1)}, \quad &\omega^{(3,i)} &\coloneqq \omega_f^{(i-1)},
\end{alignedat}
\end{equation}
where we have assumed that the same rule is used for integration over each of the three edges of the triangle, up to the uniform scaling factor of $\sqrt{2}$ for the hypotenuse.

\subsection{Summation-by-parts operators}\label{sec:sbp_tri} 
The linear operators describing partial differentiation of the interpolant in \cref{eq:nodal_tensor} at each volume quadrature node as well as those describing the evaluation of such an interpolant at each facet quadrature node can be represented in terms of the matrices $\mat{D}^{(m)} \in \R^{\nvolnodes\times\nvolnodes}$ and $\mat{R}^{(\zeta)} \in \R^{\nfacnodes\times\nvolnodes}$ as
\begin{subequations}
\begin{align}
\mat{D}^{(m)}\vc{v} &= \big[(\partial\opr{I}_q\fn{v}/\partial\xi_m)(\vec{\xi}^{(1)}), \ldots,(\partial\opr{I}_q\fn{v}/\partial\xi_m)(\vec{\xi}^{(\nvolnodes)}) \big]^\T, \label{eq:interp_deriv}\\ 
\mat{R}^{(\zeta)}\vc{v} &=  \big[(\opr{I}_q\fn{v})(\vec{\xi}^{(\zeta,1)}), \ldots, (\opr{I}_q\fn{v})(\vec{\xi}^{(\zeta,\nfacnodes)})\big]^\T,\label{eq:interp_extrap}
\end{align}
\end{subequations}
where, using \cref{eq:chain_rule_tri} and \cref{eq:tri_facet_quadrature_nodes}, the entries of such matrices are obtained as
\begin{subequations}\label{eq:d_2d}
\allowdisplaybreaks
\begin{align}
D_{\sigma(\vec{\alpha}) \sigma(\vec{\beta})}^{(1)} &\coloneqq \frac{2}{1-\eta_2^{(\alpha_2)}} \frac{\dd \ell_1^{(\beta_1)}}{\dd \eta_1} (\eta_1^{(\alpha_1)}) \delta_{\alpha_2\beta_2}, \label{eq:d1_2d}\\
D_{\sigma(\vec{\alpha}) \sigma(\vec{\beta})}^{(2)} &\coloneqq \frac{1+\eta_1^{(\alpha_1)}}{1-\eta_2^{(\alpha_2)}} \frac{\dd \ell_1^{(\beta_1)}}{\dd \eta_1} (\eta_1^{(\alpha_1)}) \delta_{\alpha_2\beta_2} + \delta_{\alpha_1\beta_1}\frac{\dd \ell_2^{(\beta_2)}}{\dd \eta_2} (\eta_2^{(\alpha_2)}),\label{eq:d2_2d}
\end{align}
\end{subequations}
and
\begin{equation}\label{eq:r_2d}
\begin{gathered}
R_{i+1, \sigma(\vec{\beta})}^{(1)} \coloneqq \ell_1^{(\beta_1)}(\eta_f^{(i)})\ell_2^{(\beta_2)}(-1),\quad 
R_{i+1, \sigma(\vec{\beta})}^{(2)} \coloneqq \ell_1^{(\beta_1)}(1)\ell_2^{(\beta_2)}(\eta_f^{(i)}),\\
R_{i+1, \sigma(\vec{\beta})}^{(3)} \coloneqq \ell_1^{(\beta_1)}(-1)\ell_2^{(\beta_2)}(\eta_f^{(i)}).
\end{gathered}
\end{equation} 
Defining diagonal matrices $\mat{W} \in \R^{\nvolnodes\times\nvolnodes}$ and $\mat{B}^{(\zeta)} \in \R^{\nfacnodes\times\nfacnodes}$ with entries given in terms of the volume and facet quadrature weights, respectively, as $W_{ij} \coloneqq \omega^{(i)}\delta_{ij}$ and $B_{ij} \coloneqq \omega^{(\zeta,i)}\delta_{ij}$, the following theorem provides sufficient conditions under which the proposed operators satisfy the requirements of \cref{def:sbp}.

\begin{theorem}\label{thm:sbp_tri}
Assuming that the one-dimensional quadrature rules in \cref{eq:quadrature_2d} satisfy \cref{eq:quad_1d} with $\tau_1^{(0,0)} \geq 2q_1$ and $\tau_2^{(0,0)} \geq 2q_2$, and that the facet quadrature rule satisfies \cref{eq:tri_facet_quad} with $\tau_f^{(0,0)} \geq 2\max(q_1,q_2)$, 
the matrices in \cref{eq:d1_2d} and \cref{eq:d2_2d} are diagonal-norm SBP operators of degree $q$, with boundary operators given as in \cref{eq:e_decomp}.
\end{theorem}
\begin{proof}
The accuracy conditions in \cref{eq:accuracy} and \cref{eq:facet_accuracy} in \cref{def:sbp} follow from \cref{lem:accuracy_2d} and the polynomial exactness of the facet quadrature rules, while the positive-definiteness of $\mat{W}$ results from the fact that the weights in \cref{eq:legendre_weights_tri} are positive when there is no node at $\eta_2 = 1$, as implied by \cref{eq:quadrature_2d}. It therefore remains to show that the SBP property is satisfied. Beginning with the $\xi_1$ direction, we define the matrix $\smash{\mat{Q}^{(1)} \coloneqq \mat{W}\mat{D}^{(1)}}$ and note that the denominator in \cref{eq:d1_2d} is cancelled by the factor $\smash{1-\eta_2^{(\alpha_2)}}$ appearing in \cref{eq:legendre_weights_tri}. Using the cardinal property $\ell_m^{(i)}(\eta_m^{(j)}) = \delta_{ij}$ of the Lagrange basis and the polynomial exactness of the quadrature, we then obtain
\begin{subequations}
\allowdisplaybreaks
\begin{align}
Q_{\sigma(\vec{\alpha})\sigma(\vec{\beta})}^{(1)} &= \underbrace{\int_{-1}^1 \ell_1^{(\alpha_1)}(\eta_1)\frac{\dd \ell_1^{(\beta_1)}}{\dd \eta_1}(\eta_1)\,\dd\eta_1}_{\tau_1^{(0,0)} \, \geq\,  2q_1 - 1 }\underbrace{\int_{-1}^1 \ell_2^{(\alpha_2)}(\eta_2) \ell_2^{(\beta_2)}(\eta_2)\,\dd\eta_2}_{\tau_2^{(0,0)} \, \geq\, 2q_2},\label{eq:q1_tri}\\
E_{\sigma(\vec{\alpha})\sigma(\vec{\beta})}^{(1)} &=\ell_1^{(\alpha_1)}\ell_1^{(\beta_1)}\big\rvert_{-1}^1\underbrace{\int_{-1}^1 \ell_2^{(\alpha_2)}(\eta_f) \ell_2^{(\beta_2)}(\eta_f)\,\dd\eta_f}_{\tau_f^{(0,0)} \, \geq \, 2q_2},
\end{align}
\end{subequations}
where the latter expression follows from substitution of \cref{eq:r_2d} and \cref{eq:normals_tri} into \cref{eq:e_decomp}, and the SBP property results from a straightforward application of integration by parts to obtain $Q_{\sigma(\vec{\alpha})\sigma(\vec{\beta})}^{(1)} = E_{\sigma(\vec{\alpha})\sigma(\vec{\beta})}^{(1)} - Q_{\sigma(\vec{\beta})\sigma(\vec{\alpha})}^{(1)}$. Similarly, considering the $\xi_2$ direction and noting that a cancellation of $1-\eta_2^{(\alpha_2)}$ also occurs for $\mat{Q}^{(2)} \coloneqq \mat{W}\mat{D}^{(2)}$, we obtain
\begin{subequations}
\allowdisplaybreaks
\begin{align}
Q_{\sigma(\vec{\alpha})\sigma(\vec{\beta})}^{(2)} =& \underbrace{\int_{-1}^1 \frac{1+\eta_1}{2}\ell_1^{(\alpha_1)}(\eta_1)\frac{\dd \ell_1^{(\beta_1)}}{\dd \eta_1}(\eta_1)\,\dd\eta_1}_{\tau_1^{(0,0)} \, \geq\,  2q_1 }\underbrace{\int_{-1}^1 \ell_2^{(\alpha_2)}(\eta_2) \ell_2^{(\beta_2)}(\eta_2)\,\dd\eta_2}_{\tau_2^{(0,0)} \, \geq\, 2q_2}\label{eq:q2_tri}\\
&+ \underbrace{\int_{-1}^1 \ell_1^{(\alpha_1)}(\eta_1)\ell_1^{(\beta_1)}(\eta_1)\,\dd\eta_1}_{\tau_1^{(0,0)} \, \geq\,  2q_1 }\underbrace{\int_{-1}^1 \frac{1-\eta_2}{2}\ell_2^{(\alpha_2)}(\eta_2) \frac{\partial\ell_2^{(\beta_2)}}{\partial \eta_2}(\eta_2)\,\dd\eta_2}_{\tau_2^{(0,0)} \, \geq\, 2q_2},\notag \\
E_{\sigma(\vec{\alpha})\sigma(\vec{\beta})}^{(2)} =& \ \ell_1^{(\alpha_1)}(1)\ell_1^{(\beta_1)}(1)\underbrace{\int_{-1}^1 \ell_2^{(\alpha_2)}(\eta_f) \ell_2^{(\beta_2)}(\eta_f)\,\dd\eta_f}_{\tau_f^{(0,0)} \, \geq \, 2q_2 } \\* &- \ell_2^{(\alpha_2)}(-1)\ell_2^{(\beta_2)}(-1)\underbrace{\int_{-1}^1 \ell_1^{(\alpha_1)}(\eta_f)\ell_1^{(\beta_1)}(\eta_f)\,\dd\eta_f}_{\tau_f^{(0,0)} \, \geq \, 2q_1}. \notag
\end{align}
\end{subequations}
Applying integration by parts and the product rule to \cref{eq:q2_tri}, we obtain $Q_{\sigma(\vec{\alpha})\sigma(\vec{\beta})}^{(2)} = E_{\sigma(\vec{\alpha})\sigma(\vec{\beta})}^{(2)} - Q_{\sigma(\vec{\beta})\sigma(\vec{\alpha})}^{(2)}$, and hence the SBP property is satisfied in the $\xi_2$ direction.
\end{proof}

\begin{remark}
\label{rmk:jacobi}
The discrete derivative operators in \cref{eq:d1_2d,eq:d2_2d} are of the same form as those proposed in \cite[section 2.2.3]{sherwin_karniadakis_triangular_sem_95}. However, while the quadrature rules employed in their work are constructed as in \cref{eq:jacobi_weights_tri} in order to subsume the Jacobian determinant of the coordinate transformation appearing in integrals such as \cref{eq:integral_tri}, such choices do not, in general, result in nodal SBP operators on the reference triangle in the sense of \cref{def:sbp}. This is because the factor of $1-\eta_2^{(\alpha_2)}$ in \cref{eq:legendre_weights_tri}, which is subsumed through the use of a Jacobi weight with $(a_2,b_2) = (1,0)$ in \cref{eq:jacobi_weights_tri}, is precisely what leads to the cancellation of the denominators in \cref{eq:d1_2d,eq:d2_2d} when $\mat{D}^{(1)}$ and $\mat{D}^{(2)}$ are pre-multiplied by $\mat{W}$, enabling the exact evaluation of the integrals in \cref{eq:q1_tri,eq:q2_tri}, respectively. While the nodal operators for $(a_2,b_2) = (1,0)$ remain exact for polynomials of degree $q$ as a consequence of \cref{lem:accuracy_2d}, they do not, in general, satisfy the SBP property. \Cref{thm:sbp_tri} therefore requires the use of quadrature rules based on the unit (i.e.\ Legendre) weight for both $\eta_1$ and $\eta_2$. As an example, we note that SBP operators of any polynomial degree $q$ can be constructed using LG quadrature rules with $q+1$ nodes in the $\eta_1$, $\eta_2$, and $\eta_f$ coordinate directions.
\end{remark}

\section{Tensor-product approximations on the reference tetrahedron}\label{sec:tet}
We now extend the methodology presented in \cref{sec:tri} to the three-dimensional case, wherein approximations are constructed on the reference tetrahedron given by
\begin{equation}\label{eq:reference_tetrahedron}
\hat{\Omega} \coloneqq \big\{\vec{\xi} \in [-1,1]^3 : \xi_1 + \xi_2 + \xi_3 \leq -1  \big\}.
\end{equation}
The facets (i.e.\ faces of the tetrahedron) are numbered as
\begin{equation}\label{eq:tet_facets}
\begin{alignedat}{2}
\hat{\Gamma}^{(1)} &\coloneqq \big\{ \vec{\xi} \in \hat{\Omega} : \xi_2 = -1 \big\}, \quad \hat{\Gamma}^{(2)} &&\coloneqq \big\{ \vec{\xi} \in \hat{\Omega} : \xi_1 + \xi_2 + \xi_3 = -1 \big\}, \\
\hat{\Gamma}^{(3)} &\coloneqq \big\{ \vec{\xi} \in \hat{\Omega} : \xi_1 = -1 \big\}, \quad \hat{\Gamma}^{(4)} &&\coloneqq \big\{ \vec{\xi} \in \hat{\Omega} : \xi_3 = -1 \big\},
\end{alignedat}
\end{equation}
with corresponding outward unit normal vectors 
\begin{equation}\label{eq:normals_tet}
\hat{\vec{n}}^{(1)} = \mqty[0\\ -1\\ 0], \quad \hat{\vec{n}}^{(2)} = \mqty[1/\sqrt{3}\\ 1/\sqrt{3}\\ 1/\sqrt{3}], \quad \hat{\vec{n}}^{(3)} = \mqty[-1\\ 0\\ 0], \quad \hat{\vec{n}}^{(4)} = \mqty[0\\ 0\\-1].
\end{equation}
As described, for example, in \cite[section 3.2]{karniadakis_sherwin_spectral_hp_element}, the collapsed coordinate transformation $\vec{\chi} : [-1,1]^3 \to \hat{\Omega}$ from the cube onto the tetrahedron in \cref{eq:reference_tetrahedron} shown in \cref{fig:mapping} is constructed from three successive applications of \cref{eq:tri_mapping} in order to obtain
\begin{equation}\label{eq:tet_mapping}
\vec{\chi}(\vec{\eta}) \coloneqq \mqty[\tfrac{1}{4}(1+\eta_1)(1-\eta_2)(1-\eta_2) - 1 \\ \tfrac{1}{2}(1+\eta_2)(1-\eta_3) - 1 \\ \eta_3],
\end{equation}
which has an inverse defined for $\vec{\xi} \in \hat{\Omega} \setminus \{[-1,1,-1]^\T, [-1,-1,1]^\T \}$ by
\begin{equation}
\vec{\chi}^{-1}(\vec{\xi}) \coloneqq \mqty[2(1+\xi_1)/(-\xi_2-\xi_3) - 1 \\ 2(1+\xi_2)/(1-\xi_3) - 1 \\ \xi_3].
\end{equation}
Integrals can then be evaluated in collapsed coordinates as
\begin{equation}\label{eq:integral_tet}
\int_{\hat{\Omega}}\fn{v}(\vec{\xi}) \, \dd \vec{\xi} = \int_{-1}^1 \int_{-1}^1\int_{-1}^1 \fn{v}(\vec{\chi}(\vec{\eta})) \frac{(1-\eta_2)(1-\eta_3)^2}{8} \,\dd \eta_1\dd \eta_2 \dd \eta_3,
\end{equation}
whereas partial derivatives can be evaluated using the chain rule as
\begin{subequations}\label{eq:chain_rule_tet}
\allowdisplaybreaks
\begin{align}
\frac{\partial \fn{v}}{\partial \xi_1} (\vec{\chi}(\vec{\eta})) &= \frac{4}{(1-\eta_2)(1-\eta_3)} \frac{\partial}{\partial \eta_1} \fn{v}(\vec{\chi}(\vec{\eta})), \\
\frac{\partial \fn{v}}{\partial \xi_2} (\vec{\chi}(\vec{\eta})) &= \bigg(\frac{2(1+\eta_1)}{(1-\eta_2)(1-\eta_3)} \frac{\partial}{\partial \eta_1} + \frac{2}{1-\eta_3} \frac{\partial}{\partial \eta_2}\bigg) \fn{v}(\vec{\chi}(\vec{\eta})), \\
\frac{\partial \fn{v}}{\partial \xi_3} (\vec{\chi}(\vec{\eta})) &= \bigg(\frac{2(1+\eta_1)}{(1-\eta_2)(1-\eta_3)} \frac{\partial}{\partial \eta_1} +  \frac{1+\eta_2}{1-\eta_3}\frac{\partial}{\partial \eta_2} + \frac{\partial}{\partial \eta_3} \bigg) \fn{v}(\vec{\chi}(\vec{\eta})).
\end{align}
\end{subequations}

\subsection{Nodal sets and interpolants}\label{sec:interp_quad_tet}
Let $q_1,q_2,q_3 \in \N_0$ denote the degrees of the approximations with respect to $\eta_1$, $\eta_2$, and $\eta_3$, respectively, with $q \coloneqq \min(q_1,q_2,q_3)$, and define Gaussian quadrature rules with (distinct) nodes and weights given by
\begin{equation}\label{eq:vol_quad_tet}
\begin{alignedat}{5}
\{\eta_1^{(i)}\}_{i \in \{0:q_1\}} &\subset [-1,1], \quad &&\{\eta_2^{(i)}\}_{i \in \{0:q_2\}} &&\subset [-1,1), \quad && \{\eta_3^{(i)}\}_{i \in \{0:q_3\}} &&\subset [-1,1),\\
\{\omega_1^{(i)}\}_{i\in\{0:q_1\}} &\subset \R^+, \quad &&\{\omega_2^{(i)}\}_{i\in\{0:q_2\}} &&\subset \R^+, \quad && \{\omega_3^{(i)}\}_{i\in\{0:q_3\}} &&\subset \R^+,
\end{alignedat}
\end{equation}
satisfying accuracy conditions analogous to \cref{eq:quad_1d} for $m \in \{1:3\}$. Defining Lagrange bases $\{\ell_m^{(i)}\}_{i \in \{0:q_m\}}$ as in \cref{eq:lagrange}, we obtain the three-dimensional interpolant
\begin{equation}\label{eq:nodal_tensor_tet}
\begin{multlined}
(\opr{I}_q \fn{v})(\vec{\chi}(\vec{\eta})) \coloneqq \sum_{\alpha_1=0}^{q_1}\sum_{\alpha_2=0}^{q_2} \sum_{\alpha_3=0}^{q_3}\ \fn{v}(\vec{\chi}(\eta_1^{(\alpha_1)}, \eta_2^{(\alpha_2)}, \eta_3^{(\alpha_3)})) \\ \times \ell_1^{(\alpha_1)}(\eta_1)\ell_2^{(\alpha_2)}(\eta_2)\ell_3^{(\alpha_3)}(\eta_3),
\end{multlined}
\end{equation}
whose degree $q$ polynomial exactness is characterized with the following lemma.

\begin{lemma}\label{lem:accuracy_3d}
The interpolation operator in \cref{eq:nodal_tensor_tet} is exact for any polynomial $\fn{v} \in \P_q(\hat{\Omega})$, satisfying
$(\opr{I}_q \fn{v})(\vec{\xi}) = \fn{v}(\vec{\xi})$ for all $\vec{\xi} \in  \hat{\Omega} \setminus \{[-1,1,-1]^\T,[-1,-1,1]^\T\}$.
\end{lemma}
\begin{proof}
The proof is similar to that of \cref{lem:accuracy_2d}, wherein we express the monomial $\pi^{(\alpha_1,\alpha_2, \alpha_3)}(\vec{\xi}) \coloneqq \xi_1^{\alpha_1}\xi_2^{\alpha_2}\xi_3^{\alpha_3}$ with $\vec{\alpha} \in \set{N}(q)$ in terms of the collapsed coordinate system under the mapping in \cref{eq:tet_mapping} to obtain
\begin{equation}\label{eq:monomial_3d}
\begin{multlined}
\pi^{(\alpha_1,\alpha_2,\alpha_3)}(\vec{\chi}(\vec{\eta})) = \big(\tfrac{1}{4}(1+\eta_1)(1-\eta_2)(1-\eta_3)-1\big)^{\alpha_1}\\ \times \big(\tfrac{1}{2}(1+\eta_2)(1-\eta_3)-1\big)^{\alpha_2} \eta_3^{\alpha_3},
\end{multlined}
\end{equation}
which is of maximum degree $\alpha_1$ in $\eta_1$, $\alpha_1 + \alpha_2$ in $\eta_2$, and $\alpha_1+\alpha_2+\alpha_3$ in $\eta_3$. The accuracy of the interpolation then follows from the fact that $\alpha_1 \leq q_1$, $\alpha_1 + \alpha_2 \leq q_2$, and $\alpha_1 + \alpha_2 + \alpha_3 \leq q_3$ for all $\vec{\alpha} \in \set{N}(q)$, placing any function $\fn{v} \circ \vec{\chi}$ which can be expressed as a linear combination of such monomials within the span of the tensor-product Lagrange basis used for the interpolation in \cref{eq:nodal_tensor_tet}.
\end{proof}

\subsection{Volume quadrature}
Let $\sigma : \{0:q_1\} \times \{0:q_2\} \times \{0:q_3\} \to \{1:\nvolnodes\}$ denote a bijective mapping inducing an ordering of the $\nvolnodes \coloneqq (q_1 + 1)(q_2 + 1)(q_3+1)$ volume quadrature nodes, which are given on the tetrahedron by $\vec{\xi}^{(\sigma(\vec{\alpha}))} \coloneqq \vec{\chi}(\eta_1^{(\alpha_1)},\eta_2^{(\alpha_2)},\eta_2^{(\alpha_3)})$. If $(a_1,b_1) = (a_2,b_2) = (a_3,b_3)= (0,0)$, the corresponding quadrature weights for the approximation of \cref{eq:integral_tet} are then given by
\begin{equation}\label{eq:legendre_weights_tet}
\omega^{(\sigma(\vec{\alpha}))} \coloneqq \frac{(1-\eta_2^{(\alpha_2)})(1-\eta_3^{(\alpha_3)})^2}{8}\omega_1^{(\alpha_1)}\omega_2^{(\alpha_2)},
\end{equation}
whereas if $(a_1,b_1) = (a_2,b_2) = (0,0)$ and $(a_3,b_3) = (1,0)$, we obtain 
\begin{equation}\label{eq:jacobi_weights_tet}
\omega^{(\sigma(\vec{\alpha}))} \coloneqq \frac{(1-\eta_2^{(\alpha_2)})(1-\eta_3^{(\alpha_3)})}{8}\omega_1^{(\alpha_1)}\omega_2^{(\alpha_2)}.
\end{equation}
By a similar argument to that in \cref{rmk:jacobi}, it can be shown that, unlike the quadrature rules specified above, the choices of $(a_1,b_1) = (0,0)$, $(a_2,b_2) = (1,0)$, and $(a_3,b_3) = (2,0)$ employed in \cite[section 3.1]{sherwin_karniadakis_tetrahedra_hp_fem_96} do not, in general, result in SBP operators on the volume quadrature nodes, and are thus not considered in this work.

\subsection{Facet quadrature}\label{sec:fac_quad_tet}
As the facets of the tetrahedron are triangular, a collapsed coordinate system can be defined on each one, which we orient so as to align with the volume coordinate system. Letting $q_{f1},q_{f2} \in \N_0$, we can therefore define a tensor-product quadrature rule in terms of the one-dimensional nodes and weights
\begin{equation}\label{eq:facet_quad_tet}
\begin{alignedat}{4}
\{\eta_{f1}^{(i)}\}_{i \in \{0:q_{f1}\}} &\subset [-1,1], \quad &&\{\eta_{f2}^{(i)}\}_{i \in \{0:q_{f2}\}} &&\subset [-1,1],\\
\{\omega_{f1}^{(i)}\}_{i \in \{0:q_{f1}\}} &\subset \R_0^+, \quad &&\{\omega_{f2}^{(i)}\}_{i \in \{0:q_{f2}\}} &&\subset \R_0^+,
\end{alignedat}
\end{equation}
satisfying accuracy conditions as in \cref{eq:quad_1d} for quadrature degrees $\tau_{f1}^{(a_{f1},b_{f1})},\tau_{f2}^{(a_{f2},b_{f2})} \in \N_0$ and Jacobi weights with exponents $a_{f1},b_{f1},a_{f2},b_{f2} > -1$. Using the bijective mapping $\sigma_{f} : \{0:q_{f1}\} \times \{0:q_{f2}\} \to \{1:\nfacnodes\}$, where $\nfacnodes \coloneqq (q_{f1} + 1)(q_{f2}+1)$ denotes the number of quadrature nodes on each facet, we obtain
\begin{equation}\label{eq:tet_facet_quadrature_nodes}
\begin{alignedat}{4}
\vec{\xi}^{(1,\sigma_f(\vec{\alpha}))} &\coloneqq \vec{\chi}(\eta_{f1}^{(\alpha_1)},-1, \eta_{f2}^{(\alpha_2)}), \quad &&\vec{\xi}^{(2,\sigma_f(\vec{\alpha}))} &&\coloneqq \vec{\chi}(1,\eta_{f1}^{(\alpha_1)}, \eta_{f2}^{(\alpha_2)}), \\
\vec{\xi}^{(3,\sigma_f(\vec{\alpha}))} &\coloneqq \vec{\chi}(-1,\eta_{f1}^{(\alpha_1)}, \eta_{f2}^{(\alpha_2)}), \quad &&\vec{\xi}^{(4,\sigma_f(\vec{\alpha}))} &&\coloneqq \vec{\chi}(\eta_{f1}^{(\alpha_1)}, \eta_{f2}^{(\alpha_2)},-1),
\end{alignedat}
\end{equation}
with corresponding quadrature weights given for $(a_{f1},b_{f1}) = (a_{f2},b_{f2})= (0,0)$ as
\begin{equation}
\begin{alignedat}{4}
\omega^{(1,\sigma_f(\vec{\alpha}))} &\coloneqq \frac{1-\eta_{f2}^{(\alpha_2)}}{2}\omega_{f1}^{(\alpha_1)}\omega_{f2}^{(\alpha_2)}, \quad &&\omega^{(2,\sigma_f(\vec{\alpha}))} &&\coloneqq \frac{\sqrt{3}(1-\eta_{f2}^{(\alpha_2)})}{2}\omega_{f1}^{(\alpha_1)}\omega_{f2}^{(\alpha_2)} , \\
\omega^{(3,\sigma_f(\vec{\alpha}))} &\coloneqq \frac{1-\eta_{f2}^{(\alpha_2)}}{2}\omega_{f1}^{(\alpha_1)}\omega_{f2}^{(\alpha_2)}, \quad &&\omega^{(4,\sigma_f(\vec{\alpha}))} &&\coloneqq \frac{1-\eta_{f2}^{(\alpha_2)}}{2}\omega_{f1}^{(\alpha_1)}\omega_{f2}^{(\alpha_2)},
\end{alignedat}
\end{equation}
where we also have the option of using a Jacobi quadrature rule with $(a_{f2},b_{f2}) = (1,0)$ as in \cref{eq:jacobi_weights_tri} to subsume the factors of $1-\eta_{f2}^{(\alpha_2)}$ appearing in the weights defined above. 
\begin{remark}
Due to the fact that the nodes in \cref{eq:tet_facet_quadrature_nodes} are arranged asymmetrically on each facet, special attention must be paid to the orientation of the local coordinate systems within each element in order for the facet quadrature nodes to align in physical space. In this work, such alignment is achieved following the second algorithm suggested in \cite[section 2.3]{sherwin_karniadakis_tetrahedra_hp_fem_96}, which was originally proposed by Warburton \etal \cite{warburton_unstructured_connectivity_95}.
\end{remark}

\subsection{Summation-by-parts operators}\label{sec:sbp_tet}
Using \cref{eq:chain_rule_tet} to differentiate the interpolant in \cref{eq:nodal_tensor_tet} with respect to the reference coordinate system, we obtain the entries of the differentiation matrices $\mat{D}^{(m)} \in \R^{\nvolnodes \times \nvolnodes}$ in \cref{eq:interp_deriv} in the tetrahedral case as
\begin{subequations}\label{eq:d_3d}
\allowdisplaybreaks
\begin{align}
D_{\sigma(\vec{\alpha}) \sigma(\vec{\beta})}^{(1)} \coloneqq& \frac{4}{(1-\eta_2^{(\alpha_2)})(1-\eta_3^{(\alpha_3)})}\frac{\dd \ell_1^{(\beta_1)}}{\dd \eta_1} (\eta_1^{(\alpha_1)}) \delta_{\alpha_2\beta_2}\delta_{\alpha_3\beta_3},\label{eq:d1_3d} \\
D_{\sigma(\vec{\alpha}) \sigma(\vec{\beta})}^{(2)} \coloneqq& \frac{2(1+\eta_1^{(\alpha_1)})}{(1-\eta_2^{(\alpha_2)})(1-\eta_3^{(\alpha_3)})}  \frac{\dd \ell_1^{(\beta_1)}}{\dd \eta_1} (\eta_1^{(\alpha_1)}) \delta_{\alpha_2\beta_2} \delta_{\alpha_3\alpha_3} \\* &+ \frac{2}{1-\eta_3^{(\alpha_3)}}\delta_{\alpha_1\beta_1}\frac{\dd \ell_2^{(\beta_2)}}{\dd \eta_2} (\eta_2^{(\alpha_2)})\delta_{\alpha_3\beta_3} \notag, \\
D_{\sigma(\vec{\alpha}) \sigma(\vec{\beta})}^{(3)} \coloneqq&  \frac{2(1+\eta_1^{(\alpha_1)})}{(1-\eta_2^{(\alpha_2)})(1-\eta_3^{(\alpha_3)})}  \frac{\dd \ell_1^{(\beta_1)}}{\dd \eta_1} (\eta_1^{(\alpha_1)}) \delta_{\alpha_2\beta_2} \delta_{\alpha_3\alpha_3}\\*  &+ \frac{1+\eta_2^{(\alpha_2)}}{1-\eta_3^{(\alpha_3)}}\delta_{\alpha_1\beta_1}\frac{\dd \ell_2^{(\beta_2)}}{\dd \eta_2} (\eta_2^{(\alpha_2)})\delta_{\alpha_3\beta_3} + \delta_{\alpha_1\beta_1} \delta_{\alpha_2\beta_2} \frac{\dd \ell_3^{(\beta_3)}}{\dd \eta_3} (\eta_3^{(\alpha_3)}). \notag
\end{align}
\end{subequations}
The entries of $\mat{R}^{(\zeta)} \in \R^{\nvolnodes\times\nfacnodes}$ in \cref{eq:interp_extrap} are likewise obtained using \cref{eq:facet_quad_tet} as
\begin{subequations}\label{eq:r_3d}
\begin{alignat}{2}
R_{\sigma_f(\vec{\alpha})\sigma(\vec{\beta})}^{(1)} &\coloneqq \ell_1^{(\beta_1)}(\eta_{f1}^{(\alpha_1)}) \ell_2^{(\beta_2)}(-1)\ell_3^{(\beta_3)}(\eta_{f2}^{(\alpha_2)}),  \\
R_{\sigma_f(\vec{\alpha})\sigma(\vec{\beta})}^{(2)}  &\coloneqq \ell_1^{(\beta_1)}(1)\ell_2^{(\beta_2)}(\eta_{f1}^{(\alpha_1)})\ell_3^{(\beta_3)}(\eta_{f2}^{(\alpha_2)}), \\
R_{\sigma_f(\vec{\alpha})\sigma(\vec{\beta})}^{(3)}  &\coloneqq \ell_1^{(\beta_1)}(-1)\ell_2^{(\beta_2)}(\eta_{f1}^{(\alpha_1)})\ell_3^{(\beta_3)}(\eta_{f2}^{(\alpha_2)}), \\
R_{\sigma_f(\vec{\alpha})\sigma(\vec{\beta})}^{(4)}  &\coloneqq \ell_1^{(\beta_1)}(\eta_{f1}^{(\alpha_1)})\ell_2^{(\beta_2)}(\eta_{f2}^{(\alpha_2)}) \ell_3^{(\beta_3)}(-1),
\end{alignat}
\end{subequations}  
resulting in the following three-dimensional analogue of \cref{thm:sbp_tri}.

\begin{theorem}\label{thm:sbp_tet}
Assuming that the volume quadrature rules in \cref{eq:vol_quad_tet} satisfy \eqref{eq:quad_1d} with $\tau_1^{(0,0)} \geq 2q_1$, $\tau_2^{(0,0)} \geq 2q_2 + 1$, and either $\tau_3^{(0,0)} \geq 2q_3 + 1$ or $\tau_3^{(1,0)} \geq 2q_3$, and that the facet quadrature rules satisfy analogous conditions with $\tau_{f1}^{(0,0)}\geq 2\max(q_1,q_2)$ and either $\tau_{f2}^{(0,0)} \geq 2\max(q_2,q_3) + 1$ or $\tau_{f2}^{(1,0)} \geq 2\max(q_2,q_3)$, the matrices in \cref{eq:d_3d} are diagonal-norm SBP operators of degree $q$, with boundary operators given as in \cref{eq:e_decomp}.
\end{theorem}
\begin{proof}
The proof is similar to that of \cref{thm:sbp_tri}, with the accuracy conditions in \cref{eq:accuracy} resulting from \cref{lem:accuracy_3d}. As for the SBP property, we define $\mat{Q}^{(1)} \coloneqq \mat{W}\mat{D}^{(1)}$, which, noting the cancellation of the factor $\smash{(1-\eta_2^{(\alpha_2)})(1-\eta_3^{(\alpha_3)})}$, results in
\begin{equation}\label{eq:q1_tet}
Q_{\sigma(\vec{\alpha})\sigma(\vec{\beta})}^{(1)} =
\underbrace{\int_{-1}^{1} \ell_1^{(\alpha_1)}\frac{\dd\ell_1^{(\beta_1)}}{\dd\eta_1}\, \dd \eta_1}_{\tau_1^{(0,0)} \, \geq \, 2q_1 - 1}\underbrace{\int_{-1}^1 \ell_2^{(\alpha_2)}\ell_2^{(\beta_2)}\,\dd \eta_2}_{\tau_2^{(0,0)} \,\geq\, 2q_2} \underbrace{\int_{-1}^1 \frac{1-\eta_3}{2} \ell_3^{(\alpha_3)} \ell_3^{(\beta_3)}\, \dd \eta_3}_{\tau_3^{(0,0)} \, \geq \, 2q_3+1 \ \text{or} \ \tau_3^{(1,0)} \, \geq \, 2q_3 }, 
\end{equation}
where the dependence of the Lagrange polynomials on the variable of integration has been suppressed within such one-dimensional integral factors for a clearer presentation, and we have used the cardinal property of the Lagrange basis and the polynomial exactness of the quadrature rules to obtain such an expression. Expressing the boundary operator $\mat{E}^{(1)}$ given in \cref{eq:e_decomp} in terms of the interpolation/extrapolation operators in \cref{eq:r_3d} as well as the normals in \cref{eq:normals_tet}, and using the exactness of the facet quadrature rules in \cref{sec:fac_quad_tet}, we likewise obtain
\begin{equation}
E_{\sigma(\vec{\alpha})\sigma(\vec{\beta})}^{(1)} = \ell_1^{(\alpha_1)}\ell_1^{(\beta_1)}\big\lvert_{-1}^1 \underbrace{\int_{-1}^{1}\ell_2^{(\alpha_2)} \ell_2^{(\beta_2)} \, \dd \eta_{f1}}_{\tau_{f1}^{(0,0)} \, \geq \, 2q_2} \underbrace{\int_{-1}^1\frac{1-\eta_{f2}}{2}\ell_3^{(\alpha_3)} \ell_3^{(\beta_3)} \, \dd\eta_{f2}}_{\tau_{f2}^{(0,0)} \, \geq \, 2q_3+1 \ \text{or} \ \tau_{f2}^{(1,0)} \, \geq \, 2q_3 }.
\end{equation}
The SBP property in the $\xi_1$ direction then follows from applying integration by parts to the first factor in \cref{eq:q1_tet}. Similarly, defining $\mat{Q}^{(2)} \coloneqq \mat{W}\mat{D}^{(2)}$ results in
\begin{multline}\label{eq:q2_tet}
Q_{\sigma(\vec{\alpha})\sigma(\vec{\beta})}^{(2)} = \\ \underbrace{\int_{-1}^{1} \frac{1+\eta_1}{2}\ell_1^{(\alpha_1)}\frac{\dd\ell_1^{(\beta_1)}}{\dd\eta_1}\, \dd \eta_1}_{\tau_1^{(0,0)} \, \geq \, 2q_1}\underbrace{\int_{-1}^1 \ell_2^{(\alpha_2)}\ell_2^{(\beta_2)}\,\dd \eta_2}_{\tau_2^{(0,0)}\, \geq \, 2q_2 }\underbrace{\int_{-1}^1 \frac{1-\eta_3}{2} \ell_3^{(\alpha_3)} \ell_3^{(\beta_3)}\, \dd \eta_3}_{\tau_3^{(0,0)} \, \geq \, 2q_3+1 \ \text{or} \ \tau_3^{(1,0)} \, \geq \, 2q_3 }\\
+ \underbrace{\int_{-1}^{1}\ell_1^{(\alpha_1)}\ell_1^{(\beta_1)}\, \dd \eta_1}_{\tau_1^{(0,0)} \, \geq \, 2q_1}\underbrace{\int_{-1}^1 \frac{1-\eta_2}{2}\ell_2^{(\alpha_2)}\frac{\dd\ell_2^{(\beta_2)}}{\dd \eta_2}\,\dd \eta_2}_{\tau_2^{(0,0)} \, \geq \, 2q_2 }\underbrace{\int_{-1}^1 \frac{1-\eta_3}{2} \ell_3^{(\alpha_3)} \ell_3^{(\beta_3)}\, \dd \eta_3}_{\tau_3^{(0,0)} \, \geq \, 2q_3+1 \ \text{or} \ \tau_3^{(1,0)} \, \geq \, 2q_3 }
\end{multline}
and
\begin{multline}\label{eq:e2_tet}
\allowdisplaybreaks
E_{\sigma(\vec{\alpha})\sigma(\vec{\beta})}^{(2)} = \ell_1^{(\alpha_1)}(1)\ell_1^{(\beta_1)}(1) \underbrace{\int_{-1}^{1}\ell_2^{(\alpha_2)} \ell_2^{(\beta_2)} \, \dd \eta_{f1}}_{\tau_{f1}^{(0,0)} \, \geq \, 2q_2}\underbrace{\int_{-1}^1\frac{1-\eta_{f2}}{2}\ell_3^{(\alpha_3)} \ell_3^{(\beta_3)} \, \dd\eta_{f2}}_{\tau_{f2}^{(0,0)} \, \geq \, 2q_3+1 \ \text{or} \ \tau_{f2}^{(1,0)} \, \geq \, 2q_3 } \\- \ell_2^{(\alpha_2)}(-1)\ell_2^{(\beta_2)}(-1)\underbrace{\int_{-1}^{1}\ell_1^{(\alpha_1)}\ell_1^{(\beta_1)}\, \dd \eta_{f1}}_{\tau_{f1}^{0,0} \, \geq \, 2q_1}\underbrace{\int_{-1}^1\frac{1-\eta_{f2}}{2}\ell_3^{(\alpha_3)} \ell_3^{(\beta_3)} \, \dd\eta_{f2}}_{\tau_{f2}^{(0,0)} \, \geq \, 2q_3+1 \ \text{or} \ \tau_{f2}^{(1,0)} \, \geq \, 2q_3 },
\end{multline}
with the SBP property in the $\xi_2$ direction resulting from the application of integration by parts and the product rule to \cref{eq:q2_tet}. Finally, one can similarly show that the SBP property in the $\xi_3$ direction follows from expressing the entries of $\mat{Q}^{(3)} \coloneqq \mat{W}\mat{D}^{(3)}$ as
\begin{multline}
Q_{\sigma(\vec{\alpha})\sigma(\vec{\beta})}^{(3)} =\\ \underbrace{\int_{-1}^{1} \frac{1+\eta_1}{2}\ell_1^{(\alpha_1)}\frac{\dd\ell_1^{(\beta_1)}}{\dd\eta_1}\, \dd \eta_1}_{\tau_1^{(0,0)} \, \geq \, 2q_1}\underbrace{\int_{-1}^1 \ell_2^{(\alpha_2)}\ell_2^{(\beta_2)}\,\dd \eta_2}_{\tau_2^{(0,0)}\, \geq \, 2q_2 }\underbrace{\int_{-1}^1 \frac{1-\eta_3}{2} \ell_3^{(\alpha_3)} \ell_3^{(\beta_3)}\, \dd \eta_3}_{\tau_3^{(0,0)} \, \geq \, 2q_3+1 \ \text{or} \ \tau_3^{(1,0)} \, \geq \, 2q_3 }   \\
+  \underbrace{\int_{-1}^{1}\ell_1^{(\alpha_1)}\ell_1^{(\beta_1)}\, \dd \eta_1}_{\tau_1^{(0,0)} \, \geq \, 2q_1}\underbrace{\int_{-1}^1\frac{(1-\eta_2)(1+\eta_2)}{4}\ell_2^{(\alpha_2)}\frac{\dd\ell_2^{(\beta_2)}}{\dd \eta_2}\,\dd \eta_2}_{\tau_2^{(0,0)} \, \geq \, 2q_2 + 1}  \underbrace{\int_{-1}^1 \frac{1-\eta_3}{2} \ell_3^{(\alpha_3)} \ell_3^{(\beta_3)}\, \dd \eta_3}_{\tau_3^{(0,0)} \, \geq \, 2q_3+1 \ \text{or} \ \tau_3^{(1,0)} \, \geq \, 2q_3 } \\ +  \underbrace{\int_{-1}^{1}\ell_1^{(\alpha_1)}\ell_1^{(\beta_1)}\, \dd \eta_1}_{\tau_1^{(0,0)} \, \geq \, 2q_1} \underbrace{\int_{-1}^1 \frac{1-\eta_2}{2}\ell_2^{(\alpha_2)}\ell_2^{(\beta_2)}\, \dd \eta_2}_{\tau_2^{(0,0)} \, \geq \, 2q_2 + 1}  \underbrace{\int_{-1}^1 \frac{(1-\eta_3)^2}{4} \ell_3^{(\alpha_3)} \frac{\dd\ell_3^{(\beta_3)}}{\dd \eta_3}\, \dd \eta_3}_{\tau_3^{(0,0)} \, \geq \, 2q_3+1 \ \text{or} \ \tau_3^{(1,0)} \, \geq \, 2q_3}
\end{multline}
and those of the corresponding boundary operator as
\begin{multline}\label{eq:e3_tet}
E_{\sigma(\vec{\alpha})\sigma(\vec{\beta})}^{(3)} = \ell_1^{(\alpha_1)}(1)\ell_1^{(\beta_1)}(1) \underbrace{\int_{-1}^{1}\ell_2^{(\alpha_2)} \ell_2^{(\beta_2)} \, \dd \eta_{f1}}_{\tau_{f1}^{(0,0)} \, \geq \, 2q_2}\underbrace{\int_{-1}^{1}\frac{1-\eta_{f2}}{2}\ell_3^{(\alpha_3)} \ell_3^{(\beta_3)} \, \dd \eta_{f2}}_{\tau_{f2}^{(0,0)} \, \geq \, 2q_3+1 \ \text{or} \ \tau_{f2}^{(1,0)} \, \geq \, 2q_3 } \\ - \ell_3^{(\alpha_3)}(-1)\ell_3^{(\beta_3)}(-1)\underbrace{\int_{-1}^{1}\ell_1^{(\alpha_1)} \ell_1^{(\beta_1)} \, \dd \eta_{f1}}_{\tau_{f1}^{(0,0)} \, \geq \, 2q_1}  \underbrace{\int_{-1}^{1}\frac{1-\eta_{f2}}{2}\ell_2^{(\alpha_2)} \ell_2^{(\beta_2)} \, \dd \eta_{f2}}_{\tau_{f2}^{(0,0)} \, \geq \, 2q_2+1 \ \text{or} \ \tau_{f2}^{(1,0)} \, \geq \, 2q_2 }.
\end{multline}
The conditions in \cref{eq:facet_accuracy} then follow from the fact that polynomials of degree $q$ are extrapolated exactly to the boundaries as in \cref{eq:extrap_acc} as a consequence of \cref{lem:accuracy_3d} and the fact that facet quadrature rules satisfying the stated assumptions are exact for the traces of polynomials of degree $2q$, which can be shown by expressing such polynomials in terms of the $(\eta_{f1},\eta_{f2})$ coordinate system in a similar manner to \cref{eq:monomial}.
\end{proof}
\begin{remark}
As an example of a set of quadrature rules satisfying the requirements of \cref{thm:sbp_tet}, we note that LG quadrature rules in the $\eta_1$, $\eta_2$, and $\eta_{f1}$ coordinates and JG quadrature rules with $(a,b) = (1,0)$ in the $\eta_3$ and $\eta_{f2}$ coordinates result in SBP operators of degree $q$ when $q+1$ nodes are used in each direction.
\end{remark}

\section{Discontinuous spectral-element methods on curvilinear meshes}\label{sec:dsem}
The operators introduced in \cref{sec:tri,sec:tet} are applied in this section to the construction of energy-stable DSEM approximations on curvilinear meshes. As a model problem, we consider a scalar conservation law governing the evolution of a quantity $\fn{u}(\vec{x},t) \in \R$ on the domain $\Omega \subset \R^d$ over the time interval $(0,T) \subset \R_0^+$, taking the form
\begin{subequations}
\begin{alignat}{2}
\frac{\partial\fn{u}(\vec{x},t)}{\partial t} + \nabla_{\vec{x}} \cdot \fvec{f}(\fn{u}(\vec{x},t))&= 0, \qquad && \forall \, (\vec{x},t) \in \Omega \times (0,T),\label{eq:pde}\\
\fn{u}(\vec{x},0) &= \fn{u}^0(\vec{x}), \qquad && \forall \, \vec{x} \in \Omega,\label{eq:ic}
\end{alignat}
\end{subequations}
subject to appropriate boundary conditions, where $\fvec{f}(\fn{u}(\vec{x},t)) \in \R^d$ denotes the flux vector and $\fn{u}^0(\vec{x}) \in \R$ denotes the initial data. We focus particularly on the linear advection equation with periodic boundary conditions, corresponding to $\fvec{f}(\fn{u}(\vec{x},t)) = \vec{a}\fn{u}(\vec{x},t)$ for an advection velocity $\vec{a} \in \R^d$, which is assumed here to be constant. 

\subsection{Nodal and modal expansions}\label{sec:nodal_modal_expansions}
Using a mesh constructed as in \cref{sec:meshmap} and the tensor-product nodal approximations introduced in \cref{sec:tri,sec:tet}, we can approximate $\fn{u}(\vec{x},t)$ and  $\fvec{f}(\fn{u}(\vec{x},t))$ in collapsed coordinates as
\begin{subequations}
\allowdisplaybreaks
\begin{align}
\fn{u}(\fvec{x}^{(\kappa)}(\vec{\chi}(\vec{\eta})),t) &\approx  \sum_{\alpha_1=0}^{q_1} \cdots \sum_{\alpha_d=0}^{q_d} u_{\sigma(\vec{\alpha})}^{(h,\kappa)}(t) \ell_1^{(\alpha_1)}(\eta_1) \cdots \ell_d^{(\alpha_d)}(\eta_d),\label{eq:nodal_solution}\\
\fvec{f}(\fn{u}(\fvec{x}^{(\kappa)}(\vec{\chi}(\vec{\eta})),t)) &\approx  \sum_{\alpha_1=0}^{q_1} \cdots \sum_{\alpha_d=0}^{q_d} \fvec{f}(u_{\sigma(\vec{\alpha})}^{(h,\kappa)}(t)) \ell_1^{(\alpha_1)}(\eta_1) \cdots \ell_d^{(\alpha_d)}(\eta_d).\label{eq:nodal_flux}
\end{align}
\end{subequations}
Although it is often advantageous from an efficiency perspective to collocate the solution degrees of freedom with the volume quadrature points by directly evolving the nodal solution vector $\vc{u}^{(h,\kappa)}(t) \in \R^{\nvolnodes}$, there are two disadvantages to such an approach within the present context. First, the number of tensor-product quadrature nodes $\nvolnodes$ can be much larger than the cardinality $N_q^*$ of a total-degree polynomial basis for the space $\P_q(\hat{\Omega})$. Second, the collapsed edge introduces a clustering of resolution near the singularity, and therefore limits the maximum stable time step size for explicit schemes.\footnote{Quoting Dubiner \cite{dubiner_spectral_triangle_91} concerning the time step restriction for tensor-product polynomials in collapsed coordinates, ``[resolution] is a good thing, but this is a case of too much of a good thing.''}\ We therefore propose as an alternative the use of a \emph{modal} approach, wherein a basis $\{\phi^{(i)}\}_{i\in\{1:\npoly\}}$ for $\P_p(\hat{\Omega})$ with $p \leq q$ is used to represent the solution on the reference element in terms of the modal expansion coefficients $\smash{\vc{\tilde{u}}^{(h,\kappa)}(t) \in \R^{\npoly}}$ as
\begin{equation}\label{eq:modal_expansion}
\fn{u}(\fvec{x}^{(\kappa)}(\vec{\xi}),t) \approx \sum_{i=1}^{\npoly} \tilde{u}_i^{(h,\kappa)}(t)\phi^{(i)}(\vec{\xi}),
\end{equation}
which is evaluated at each node to obtain the coefficients $\vc{u}^{(h,\kappa)}(t)$ for the nodal expansion in \eqref{eq:nodal_solution}. This mapping from modal to nodal space can be expressed in terms of the \emph{generalized Vandermonde matrix} $\mat{V} \in \R^{\nvolnodes\times\npoly}$ as
\begin{equation}\label{eq:modal_to_nodal}
\vc{u}^{(h,\kappa)}(t) = \mat{V}\vc{\tilde{u}}^{(h,\kappa)}(t), \quad 
V_{ij} \coloneqq \phi^{(j)}(\vec{\xi}^{(i)}).
\end{equation}
Note that since any polynomial $\fn{v} \in \P_p(\hat{\Omega})$ with $p \leq q$ admits a unique expansion in terms of the tensor-product Lagrange basis, the mapping in \eqref{eq:modal_to_nodal} is injective and hence $\mat{V}$ is of rank $\npoly$ (i.e.\ full column rank), regardless of the chosen basis for $\P_p(\hat{\Omega})$. We must, however, carefully choose such a basis to ensure that the cost of applying $\mat{V}$ is minimized and the tensor-product structure of the discretization is preserved.

\subsection{Proriol-Koornwinder-Dubiner basis functions}\label{sec:modal_basis}
In order to construct polynomial bases on the triangle and tetrahedron for which operations such as \cref{eq:modal_to_nodal} are amenable to sum factorization, we follow \cite[section 3.2]{karniadakis_sherwin_spectral_hp_element} and define 
\begin{equation}\label{eq:principal_functions}
\begin{gathered}
\psi_1^{(\alpha_1)}(\eta_1) \coloneqq \sqrt{2}\fn{P}_{\alpha_1}^{(0,0)}(\eta_1), \quad
\psi_2^{(\alpha_1,\alpha_2)}(\eta_2) \coloneqq (1-\eta_2)^{\alpha_1}\fn{P}_{\alpha_2}^{(2\alpha_1+1,0)}(\eta_2), \\
\psi_3^{(\alpha_1,\alpha_2,\alpha_3)}(\eta_3) \coloneqq 2(1-\eta_3)^{\alpha_1 + \alpha_2}\fn{P}_{\alpha_3}^{(2\alpha_1+2\alpha_2+2,0)}(\eta_3),
\end{gathered}
\end{equation}
in terms of the normalized Jacobi polynomials $\fn{P}_i^{(a,b)} \in \P_i([-1,1])$, which are orthonormal with respect to the weight functions $\varpi^{(a,b)}(\eta) \coloneqq (1-\eta)^{a} (1+\eta)^{b}$ considered in \cref{sec:interp_tri}, and can be constructed through recurrence relations, as shown, for example, by Hesthaven and Warburton \cite[appendix A]{hesthaven08}. The Proriol-Koornwinder-Dubiner (PKD) polynomials \cite{proriol_polynomials_57,koornwinder_orthogonal_polynomials_75,dubiner_spectral_triangle_91} are then given on the reference triangle as
\begin{equation}\label{eq:modal_basis_2d}
\phi^{(\pi(\vec{\alpha}))}(\vec{\chi}(\vec{\eta})) \coloneqq \psi_1^{(\alpha_1)}(\eta_1) \psi_2^{(\alpha_1,\alpha_2)}(\eta_2), 
\end{equation}
and on the reference tetrahedron as
\begin{equation}\label{eq:modal_basis_3d}
\phi^{(\pi(\vec{\alpha}))}(\vec{\chi}(\vec{\eta})) \coloneqq \psi_1^{(\alpha_1)}(\eta_1) \psi_2^{(\alpha_1,\alpha_2)}(\eta_2)\psi_3^{(\alpha_1,\alpha_2,\alpha_3)}(\eta_3), 
\end{equation}
where we order the multi-indices $\vec{\alpha} \in \set{N}(p)$ using the bijection $\pi : \set{N}(p) \to \{1:\npoly\}$. Although the resulting matrix $\mat{V}$ is not a standard tensor-product operator, the ``warped'' tensor-product structure of the PKD basis allows for such a matrix (as well as its transpose) to be applied through sum factorization, as described, for example, in \cite[sections 4.1.6.1 and 4.1.6.2]{karniadakis_sherwin_spectral_hp_element}. Moreover, such bases are orthonormal with respect to the $L^2$ inner product on the reference element, resulting in the reference mass matrix $\mat{M} \coloneqq \mat{V}^\T\mat{W}\mat{V}$ being the identity matrix if the quadrature rule in \cref{eq:sbp_quadrature} is of degree $\tau \geq 2p$. As will be discussed in \cref{sec:implementation}, the proposed algorithms exploit both the tensor-product structure and the orthonormality of the PKD basis. 

\subsection{Discretization using SBP operators}\label{sec:discretization}
Integrating \cref{eq:pde} by parts over a mesh element against a smooth test function $\fn{v}$, we obtain a local weak formulation, which, suppressing dependence on $\vec{x}$ and $t$ for brevity, is given by
\begin{equation}\label{eq:weak}
\int_{\Omega^{(\kappa)}} \fn{v} \frac{\partial\fn{u}}{\partial t} \, \dd\vec{x} = \int_{\Omega^{(\kappa)}}\nabla_{\vec{x}}\fn{v} \cdot \fvec{f}(\fn{u}) \, \dd\vec{x} - \int_{\partial\Omega^{(\kappa)}} \fn{v}\fn{f}^*(\fn{u}^-, \fn{u}^+, \vec{n}) \, \dd s,
\end{equation}
where, as with standard DG methods such as those described in \cite{hesthaven08}, the normal flux component $\fvec{f}(\fn{u}) \cdot \vec{n}$ has been replaced by a  \emph{numerical flux function} (see, for example, Toro \cite{toro09}). For the linear advection equation, the numerical flux takes the form
\begin{equation}\label{eq:num_flux_lf}
\fn{f}^*(\fn{u}^-, \fn{u}^+, \vec{n}) \coloneqq \frac{1}{2}\big(\fvec{f}(\fn{u}^-) + \fvec{f}(\fn{u}^+)\big)\cdot \vec{n} - \frac{\lambda}{2}\big\lvert\vec{a} \cdot \vec{n}\big\rvert (\fn{u}^+ - \fn{u}^-),
\end{equation}
recovering an upwind flux for $\lambda = 1$ and a central flux for $\lambda = 0$, where $\fn{u}^-,\fn{u}^+ \in \R$ are the internal and external solution states, respectively, and $\vec{n} \in \S^{d-1}$ denotes the outward unit normal vector. Discretizing the right-hand side of \cref{eq:weak} using SBP operators on the physical element constructed as in \cref{sec:sbp_phys}, we obtain 
\begin{equation}\label{eq:rhs}
\vc{r}^{(h,\kappa)}(t) \coloneqq \sum_{m=1}^d\big(\mat{Q}^{(\kappa,m)}\big)^\T\vc{f}^{(h,\kappa,m)}(t) - \sum_{\zeta=1}^{\nfac}\big(\mat{R}^{(\zeta)}\big)^\T \mat{B}^{(\zeta)}\mat{J}^{(\kappa,\zeta)}\vc{f}^{(*,\kappa,\zeta)}(t),
\end{equation}
where the physical flux components are evaluated at the volume quadrature nodes as
\begin{equation}\label{eq:phys_flux}
\vc{f}^{(h,\kappa,m)}(t) \coloneqq \big[\fn{f}_m(u_1^{(h,\kappa)}(t)), \ldots, \fn{f}_m(u_{\nvolnodes}^{(h,\kappa)}(t))\big]^\T,
\end{equation}
and the numerical flux is evaluated at the facet quadrature nodes in terms of the interior state $\vc{u}^{(h,\kappa,\zeta)}(t) \coloneqq \mat{R}^{(\zeta)}\vc{u}^{(h,\kappa)}(t)$ and the exterior state $\vc{u}^{(+,\kappa,\zeta)}(t) \in \R^{\nfacnodes}$ as
\begin{equation}\label{eq:num_flux}
\vc{f}^{(*,\kappa,\zeta)}(t) \coloneqq \mqty[\fn{f}^*\big(u_1^{(h,\kappa,\zeta)}(t),\, u_1^{(+,\kappa,\zeta)}(t),\, \vec{n}^{(\kappa,\zeta)}(\fvec{x}^{(\kappa)}(\vec{\xi}^{(1)}))\big) \\ \vdots \\ \fn{f}^*\big(u_{\nfacnodes}^{(h,\kappa,\zeta)}(t),\, u_{\nfacnodes}^{(+,\kappa,\zeta)}(t),\, \vec{n}^{(\kappa,\zeta)}(\fvec{x}^{(\kappa)}(\vec{\xi}^{(\nfacnodes)}))\big)].
\end{equation} 
The semi-discrete nodal and modal formulations are then given, respectively, by
\begin{subequations}
\begin{align}
\mat{W}\mat{J}^{(\kappa)}\dv{\vc{u}^{(h,\kappa)}(t)}{t}&= \vc{r}^{(h,\kappa)}(t), \label{eq:nodal}\\
\mat{M}^{(\kappa)} \dv{\vc{\tilde{u}}^{(h,\kappa)}(t)}{t} &= \mat{V}^\T\vc{r}^{(h,\kappa)}(t) \label{eq:modal},
\end{align}
\end{subequations}
where $\mat{M}^{(\kappa)} \coloneqq \mat{V}^\T\mat{W}\mat{J}^{(\kappa)}\mat{V}$ denotes the physical mass matrix. The initial condition in \cref{eq:ic} is imposed for the nodal approach by restricting $\fn{u}^0(\fvec{x}^{(\kappa)}(\vec{\xi}))$ to the volume quadrature nodes, whereas for the modal approach, we solve the projection problem
\begin{equation}\label{eq:ic_modal}
\mat{M}^{(\kappa)}\vc{u}^{(h,\kappa)}(0) = \mat{V}^\T\mat{W}\mat{J}^{(\kappa)}\big[\fn{u}^0(\fvec{x}^{(\kappa)}(\vec{\xi}^{(1)})), \ldots, \fn{u}^0(\fvec{x}^{(\kappa)}(\vec{\xi}^{(\nvolnodes)}))\big]^\T.
\end{equation}
Due to \cref{thm:sbp_tri,thm:sbp_tet} as well as the use of the split formulation in \cref{eq:q_decomp} encapsulated within $\smash{\mat{Q}^{(\kappa,m)}}$, the discretizations in \cref{eq:nodal} and \cref{eq:modal} directly inherit the accuracy, conservation, free-stream preservation, and energy stability properties afforded through the use of diagonal-norm multidimensional SBP operators (see, for example, Hicken \etal \cite{hicken_mdsbp_16} and Del Rey Fern\'andez \etal \cite{delrey_mdsbp_sat_18,delrey_extension_of_dense_gsbp_tensor_curvilinear_19}). Proofs of such results for the specific discretizations in this work follow directly from a straightforward generalization of the analysis in \cite[section 4.3]{montoya_tensor_product_22} to include tetrahedral elements.

\subsection{Weight-adjusted approximation of the mass matrix inverse}\label{sec:weight_adjusted}
The mass matrix $\mat{M}^{(\kappa)}$ appearing on the left-hand side of \eqref{eq:modal} is dense when the mapping from the reference element to the physical element is not affine, with an inverse that lacks a tensor-product structure. Obtaining the time derivative for such a scheme in the context of explicit temporal integration thus requires either the storage and application of a non-tensorial factorization or inverse, or, otherwise, the solution of a linear system. To obtain a fully explicit formulation for the time derivative in \cref{eq:modal}, we therefore adopt the following \emph{weight-adjusted} approximation from Chan \etal \cite{chan_weight_adjusted_dg_curvilinear_17}:
\begin{equation}\label{eq:weight_adjusted_inverse}
\big(\mat{M}^{(\kappa)}\big)^{-1} \approx \mat{M}^{-1}\mat{V}^\T\mat{W}\big(\mat{J}^{(\kappa)}\big)^{-1}\mat{V}  \mat{M}^{-1} =\vcentcolon \big(\mat{\tilde{M}}^{(\kappa)}\big)^{-1}.
\end{equation}
The above approximation was initially proposed for the purpose of reducing storage, but in the present context has the additional advantage of retaining the tensor-product structure which is otherwise lost by taking the inverse of the mass matrix, allowing for all operations involved in evaluating the time derivative to be performed through sum factorization. Although this work represents the first application, to the authors' knowledge, of the weight-adjusted methodology to tensor-product SEM formulations on simplices, the restoration of a tensor-product structure using a weight-adjusted approximation was previously exploited, for example, in the context of isogeometric analysis by Chan and Evans \cite{chan_evans_iga_18} as well as for Galerkin difference methods by Kozdon \etal \cite{kozdon_complex_geometries_galerkin_difference_19}. Following a similar analysis to that in \cite{chan_weight_adjusted_dg_curvilinear_17}, energy stability can then be established with respect to the discrete norm induced by the modified mass matrix $\mat{\tilde{M}}^{(\kappa)}$ under the same conditions as required for the unmodified scheme. Additionally, our implementation makes use of a further modification proposed by Chan and Wilcox \cite[Lemma 2]{chan_wilcox_entropystable_curvilinear_19}, wherein $\smash{\mat{J}^{(\kappa)}}$ is approximated using a projection of $\fn{J}^{(\kappa)}$ onto $\P_p(\hat{\Omega})$, which, while having a negligible impact on the accuracy of the scheme, ensures discrete conservation when the volume quadrature rule in \cref{eq:sbp_quadrature} is of at least degree $\tau \geq 2p$. 

\section{Efficient implementation}\label{sec:implementation}
In this section, we discuss the efficient implementation of the proposed schemes, particularly in conjunction with explicit time integration. With the exception of the numerical flux evaluation, all operations are local to a given element, and can therefore be executed in parallel in a straightforward manner. Several strategies for computing such local operations are described below.

\subsection{Reference-operator algorithms}\label{sec:reference_operator}
Expressing \cref{eq:rhs} in terms of operators on the reference element and separating the volume and facet contributions, we obtain 
\begin{equation}\label{eq:rhs_new}
\begin{multlined}
\vc{r}^{(h,\kappa)}(t) =  \sum_{l=1}^d \sum_{m=1}^d \bigg( \big(\mat{D}^{(l)}\big)^\T\Big[\tfrac{1}{2}\mat{W}\mat{\Lambda}^{(\kappa,l,m)}\Big]\vc{f}^{(h,\kappa,m)}(t) \\ -\Big[\tfrac{1}{2}\mat{W}\mat{\Lambda}^{(\kappa,l,m)}\Big]\mat{D}^{(l)}\vc{f}^{(h,\kappa,m)}(t)\bigg)  \\  -\sum_{\zeta=1}^{\nfac}\big(\mat{R}^{(\zeta)}\big)^\T \Big[\mat{B}^{(\zeta)}\mat{J}^{(\kappa,\zeta)}\Big]\bigg(\vc{f}^{(*,\kappa,\zeta)}(t) -\sum_{m=1}^d \Big[\tfrac{1}{2}\mat{N}^{(\kappa,\zeta,m)}\Big]\mat{R}^{(\zeta)}\vc{f}^{(h,\kappa,m)}(t) \bigg),
\end{multlined}
\end{equation}
where square brackets are used to denote operators which must be precomputed and stored for each element, which in \cref{eq:rhs_new} are all diagonal. For discretizations on triangles and tetrahedra using non-tensorial operators, the reference operators $\mat{D}^{(l)}$ and $\mat{R}^{(\zeta)}$ are typically stored as dense matrices and applied, for example, using standard BLAS operations. In the context of the proposed schemes, however, we have the additional option of exploiting the tensor-product structure of such operators through sum factorization. To obtain a further optimization in such cases, we redefine the following operators:
\begin{align}
\Big[\tfrac{1}{2}\mat{W}\mat{\Lambda}^{(\kappa,l,m)} \Big]_{\sigma(\vec{\alpha})\sigma(\vec{\beta})} &\gets \frac{1}{2}\sum_{n=1}^d \big[\fn{J}^{(\kappa)}(\vec{\xi}^{(\sigma(\vec{\alpha}))})(\nabla_{\vec{\xi}}\fvec{x}^{(\kappa)}(\vec{\xi}^{(\sigma(\vec{\alpha}))}))^{-1}\big]_{nm} \\ & \quad \times \big[(\nabla_{\vec{\eta}}\vec{\chi}(\eta_1^{(\alpha_1)},\ldots,\eta_d^{(\alpha_d)}) )^{-1}\big]_{ln} \omega^{(\sigma(\vec{\alpha}))}\delta_{\sigma(\vec{\alpha})\sigma(\vec{\beta})}, \notag\\
\Big[\mat{D}^{(l)}\Big]_{\sigma(\vec{\alpha})\sigma(\vec{\beta})} &\gets \frac{\dd\ell_l^{(\beta_l)}}{\dd \eta_l}(\eta_l^{(\alpha_l)})\prod_{m \in \{1:d\} \setminus \{l\}} \delta_{\alpha_m\beta_m} .
\end{align}
These modifications combine the geometric factors arising from the transformations $\vec{\chi} : [-1,1]^d \to \hat{\Omega}$ and $\fvec{x}^{(\kappa)}: \hat{\Omega} \to \Omega^{(\kappa)}$, allowing for the volume contributions in \cref{eq:rhs_new} to be evaluated in collapsed coordinates with an equivalent cost per element to that of a comparable tensor-product discretization on curved quadrilaterals or hexahedra. To evaluate the time derivative for the nodal formulation, we simply pre-multiply the right-hand-side vector by the inverse of the diagonal nodal mass matrix, resulting in
\begin{equation}\label{eq:dudt_nodal}
\dv{\vc{u}^{(h,\kappa)}(t)}{t} = \Big[\big(\mat{W}\mat{J}^{(\kappa)}\big)^{-1}\Big]\vc{r}^{(h,\kappa)}(t).
\end{equation}For the weight-adjusted modal formulation, the time derivative is given explicitly as
\begin{equation}\label{eq:dudt_modal}
\dv{\vc{\tilde{u}}^{(h,\kappa)}(t)}{t} =  \mat{M}^{-1}\mat{V}^\T\Big[\mat{W}\big(\mat{J}^{(\kappa)}\big)^{-1}\Big]\mat{V}\mat{M}^{-1}\mat{V}^\T\vc{r}^{(h,\kappa)}(t),
\end{equation}
where we recall that the application of $\mat{M}^{-1}$ can be avoided by choosing an orthonormal basis and using a volume quadrature rule of degree $2p$ or higher. Since the use of the PKD basis in \cref{sec:modal_basis} allows for $\mat{V}$ and $\mat{V}^\T$ to be applied using sum factorization, and all other operators are either diagonal or possess a standard Kronecker-product structure, the number of operations required for evaluating the time derivative in either \cref{eq:dudt_nodal} or \cref{eq:dudt_modal} scales as $O(p^{d+1})$, assuming in the triangular case that $q_1$, $q_2$, and $q_f$ scale as $O(p)$, and in the the tetrahedral case that $q_1$, $q_2$, $q_3$, $q_{f1}$, and $q_{f2}$ scale as $O(p)$. Asymptotically, this compares favourably to the $O(p^{2d})$ complexity of a standard (i.e.\ non-tensor-product) multidimensional scheme similarly employing $O(p^d)$ volume quadrature nodes and $O(p^{d-1})$ facet quadrature nodes.

\subsection{Physical-operator algorithms}\label{sec:physical_operator}
Whether or not sum factorization is employed, and whether a nodal or modal formulation is chosen, the algorithms described in \cref{sec:reference_operator} all share the feature of avoiding the precomputation and storage of dense operator matrices for each physical element. If such memory considerations are not a constraint, however, one has the option of instead precomputing physical operator matrices, an approach which can be competitive with sum factorization at lower polynomial degrees despite scaling asymptotically as $O(p^{2d})$. The time derivative can then be obtained for the nodal formulation as 
\begin{equation}
\begin{multlined}
\dv{\vc{u}^{(h,\kappa)}(t)}{t} = \sum_{m=1}^d\Big[\big(\mat{W}\mat{J}^{(\kappa)}\big)^{-1}\big(\mat{Q}^{(\kappa,m)}\big)^\T\Big] \vc{f}^{(h,\kappa,m)}(t) \\ - \sum_{\zeta=1}^{\nfac} \Big[ \big(\mat{W}\mat{J}^{(\kappa)}\big)^{-1}\big(\mat{R}^{(\zeta)}\big)^\T\mat{B}^{(\zeta)}\mat{J}^{(\kappa,\zeta)} \Big] \vc{f}^{(*,\kappa,\zeta)}(t),
\end{multlined}
\end{equation}
and for the modal formulation as
\begin{equation}
\begin{multlined}
\dv{\vc{\tilde{u}}^{(h,\kappa)}(t)}{t} = \sum_{m=1}^d\Big[\big(\mat{\tilde{M}}^{(\kappa)}\big)^{-1}\mat{V}^\T\big(\mat{Q}^{(\kappa,m)}\big)^\T\Big] \vc{f}^{(h,\kappa,m)}(t) \\ - \sum_{\zeta=1}^{\nfac}\Big[ \big(\mat{\tilde{M}}^{(\kappa)}\big)^{-1}\mat{V}^\T\big(\mat{R}^{(\zeta)}\big)^\T\mat{B}^{(\zeta)}\mat{J}^{(\kappa,\zeta)} \Big] \vc{f}^{(*,\kappa,\zeta)}(t),
\end{multlined}
\end{equation}
where, although we have used $\mat{\tilde{M}}^{(\kappa)}$ instead of $\mat{M}^{(\kappa)}$ for consistency with \eqref{eq:dudt_modal}, there is no advantage in operation count (aside from that of precomputing the inverse) nor in storage to using such an approximation for physical-operator algorithms.

\section{Numerical experiments}\label{sec:numerical}
In this section, we assess the numerical properties of the proposed split-form DSEMs on triangles and tetrahedra using tensor-product SBP operators within the context of linear advection problems on curvilinear meshes and compare such schemes to those of the same form using multidimensional SBP operators based on symmetric quadrature rules. The scripts required to run all simulations described in this section as well as the Jupyter notebooks used to generate the figures are provided within the reproducibility repository \url{https://github.com/tristanmontoya/ReproduceSBPSimplex}, employing the open-source Julia code \texttt{StableSpectralElements.jl} developed by the first author, which is available at \url{https://github.com/tristanmontoya/StableSpectralElements.jl}.
\subsection{Problem setup and curvilinear mesh generation}
We solve the linear advection equation on the domain $\Omega \coloneqq (0,1)^d$, with an advection velocity of $\vec{a} \coloneqq [1,1]^\T$ in two dimensions and $\vec{a} \coloneqq [1,1,1]^\T$ in three dimensions, where periodic boundary conditions are applied in all directions and the initial condition in \cref{eq:ic} is given by $\fn{u}^0(\vec{x}) \coloneqq \sin(2\pi x_1) \cdots \sin(2\pi x_d)$. The mesh is generated by splitting a Cartesian grid of quadrilateral or hexahedral elements with $M \in \N$ edges in each direction into triangles or tetrahedra of equal size and using an affine transformation to map the interpolation nodes for a degree $p_g \in \N$ Lagrange basis onto each element, which in the present work correspond to those generated through the interpolatory warp-and-blend procedure of Chan and Warburton \cite{chan_warburton_interpolation_nodes_15}. Following Chan \etal \cite[section 5]{chan_delrey_carpenter_gauss_collocation_19}, such a mesh is curved by perturbing the interpolation node positions as
\begin{equation}
\begin{aligned}
\tilde{x}_1 &\gets x_1 + \varepsilon \cos(\pi\big(x_1 - \tfrac{1}{2}\big) )\cos(3\pi\big(x_2 - \tfrac{1}{2} \big) ),\\
\tilde{x}_2 &\gets x_2 + \varepsilon \sin(4\pi\big(\tilde{x}_1 - \tfrac{1}{2} \big) )\cos(\pi\big(x_2 - \tfrac{1}{2}\big) ),
\end{aligned}
\end{equation}
in two dimensions, and as
\begin{equation}
\begin{aligned}
\tilde{x}_2 &\gets x_2 + \varepsilon \cos(3\pi\big(x_1 - \tfrac{1}{2} \big) )\cos(\pi\big(x_2 - \tfrac{1}{2}\big) )\cos(\pi\big(x_3 - \tfrac{1}{2} \big) ),\\
\tilde{x}_1 &\gets x_1 + \varepsilon \cos(\pi\big(x_1 - \tfrac{1}{2} \big) )\sin(4\pi\big(\tilde{x}_2 - \tfrac{1}{2}\big) ) \cos(\pi\big(x_3 - \tfrac{1}{2} \big) ),\\
\tilde{x}_3 &\gets x_3 + \varepsilon \cos(\pi\big(\tilde{x}_1 - \tfrac{1}{2} \big)) \cos(2\pi\big(\tilde{x}_2 - \tfrac{1}{2}\big) ) \cos(\pi\big(x_3 - \tfrac{1}{2} \big)),
\end{aligned}
\end{equation}
in three dimensions, where we take $\varepsilon = 1/16$ in both cases. The new node positions $\tilde{\vec{x}}$ are then used to construct the curvilinear mapping $\fvec{x}^{(\kappa)} \in [\P_{p_g}(\hat{\Omega})]^d$ from reference to physical coordinates through Lagrange interpolation, from which the geometric factors in \cref{eq:geometric_factors} can be obtained analytically. We take $p_g = 3$ in the triangular case and $p_g = 2$ in the tetrahedral case, such that the discrete metric identities in \cref{eq:discrete_metric_identities} are automatically satisfied for SBP operators of any polynomial degree $p \geq 2$. The systems of ordinary differential equations resulting from the proposed spatial discretizations of \cref{eq:pde} are then integrated in time until $T = 1$ using a Julia implementation \cite{rackauckas_julia_diffeq_17} of the five-stage, fourth-order explicit low-storage Runge-Kutta method of Carpenter and Kennedy \cite{carpenter_kennedy_rk_94}, with the time step taken to be sufficiently small for the error due to the temporal discretization to be dominated by that due to the spatial discretization.

\subsection{Multidimensional and tensor-product SBP operators}
To provide a point of comparison for the novel tensor-product discretizations presented in this paper, we construct multidimensional SBP operators using symmetric quadrature rules on the reference element as described by Chan \cite[Lemma 1]{chan_discretely_entropy_conservative_dg_sbp_18}. Specifically, for SBP operators of degree $p$ on the triangle, we use degree $2p$ Xiao-Gimbutas quadrature rules \cite{xiao_gimbutas_quadrature_10} for volume integration and degree $2p+1$ LG quadrature rules for facet integration. On the tetrahedron, we use degree $2p$ Ja\'skowiec-Sukumar quadrature rules \cite{jaskowiec_sukumar_symmetric_cubature_21} for volume integration and degree $2p$ triangular quadrature rules from \cite{xiao_gimbutas_quadrature_10} for facet integration. The resulting operators are henceforth denoted as \emph{multidimensional} to distinguish them from the \emph{tensor-product} operators introduced in this work, and are available for degrees $p \leq 25$ on the triangle and $p \leq 10$ on the tetrahedron. The tensor-product SBP operators which we construct on the triangle employ LG quadrature rules with $p+1$ nodes for integration with respect to $\eta_1$, $\eta_2$, and $\eta_f$, corresponding to $q_1 = q_2 = q_f = p$, whereas those on the tetrahedron employ LG quadrature rules for $\eta_1$ and $\eta_2$ alongside a JG quadrature rule with $(a_3,b_3) = (1,0)$ for $\eta_3$, with $p+1$ nodes in each direction (i.e.\ taking $q_1 = q_2 = q_3 = p$). The facet quadrature on the tetrahedron consists of an LG rule in the $\eta_{f1}$ direction and a JG rule with $(a_{f2},b_{f2}) = (1,0)$ in the $\eta_{f2}$ direction, where we use $p+1$ nodes in each direction, corresponding to $q_{f1} = q_{f2} = p$. As such quadrature rules satisfy the conditions of \cref{thm:sbp_tri,thm:sbp_tet}, valid SBP operators in the sense of \cref{def:sbp} are obtained for all polynomial degrees.

\subsection{Conservation and energy stability}\label{sec:cons_stab}
\begin{figure}[t!]
\centering
\begin{subfigure}{0.495\textwidth}
\centering
\includegraphics[height=36mm]{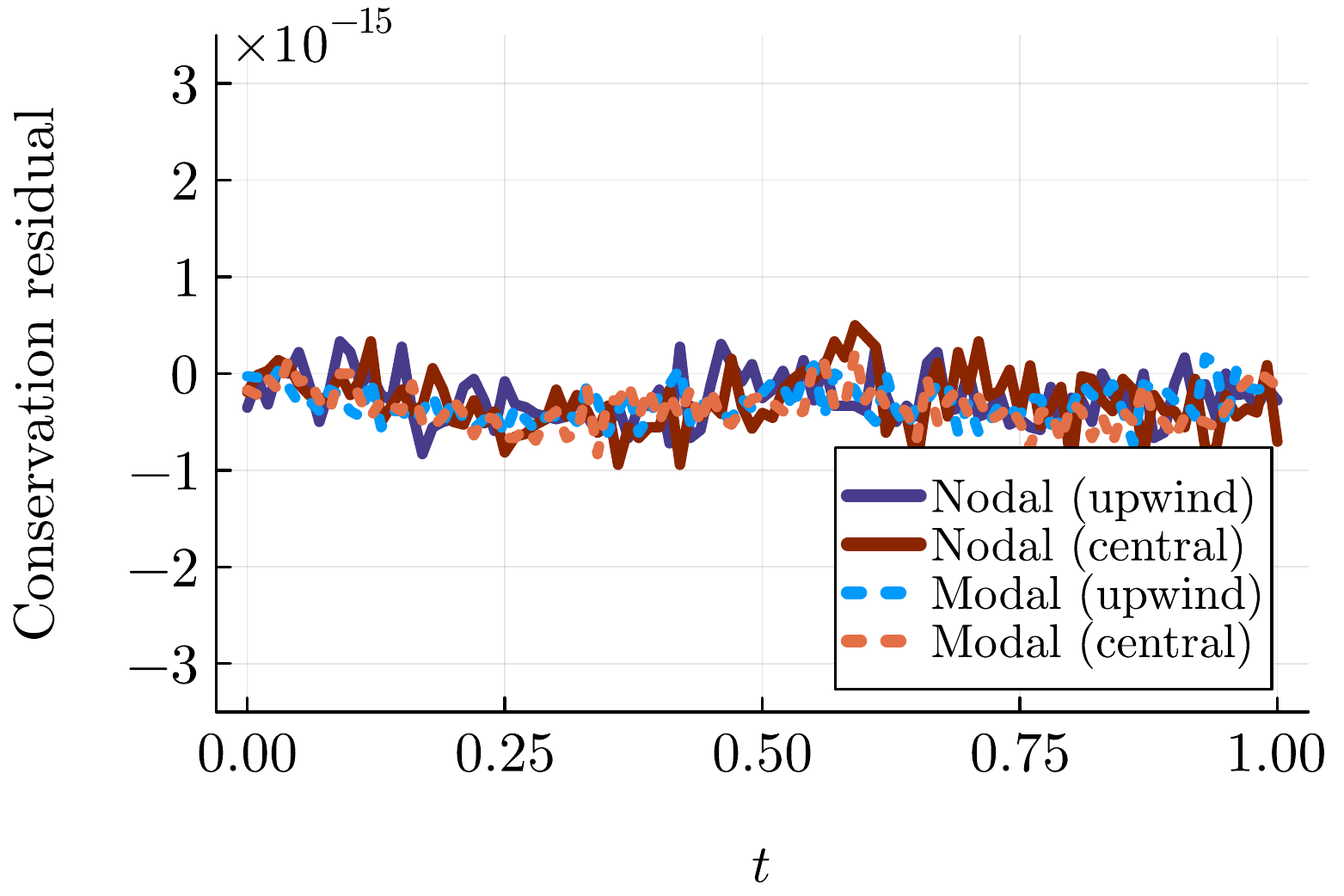}
\end{subfigure}
\begin{subfigure}{0.495\textwidth}
\centering
\includegraphics[height=36mm]{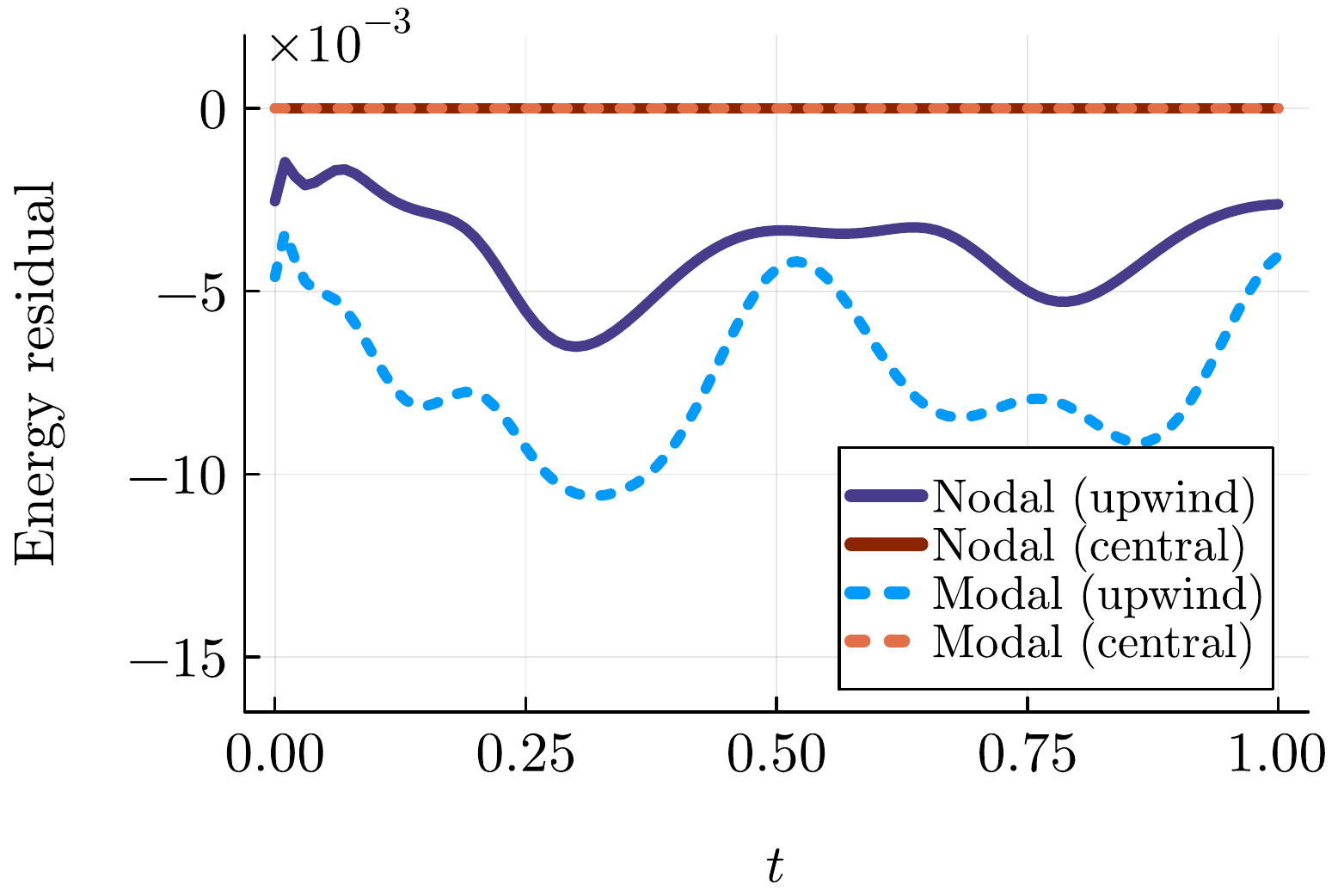}
\end{subfigure}\\
\begin{subfigure}{0.495\textwidth}
\centering
\includegraphics[height=36mm]{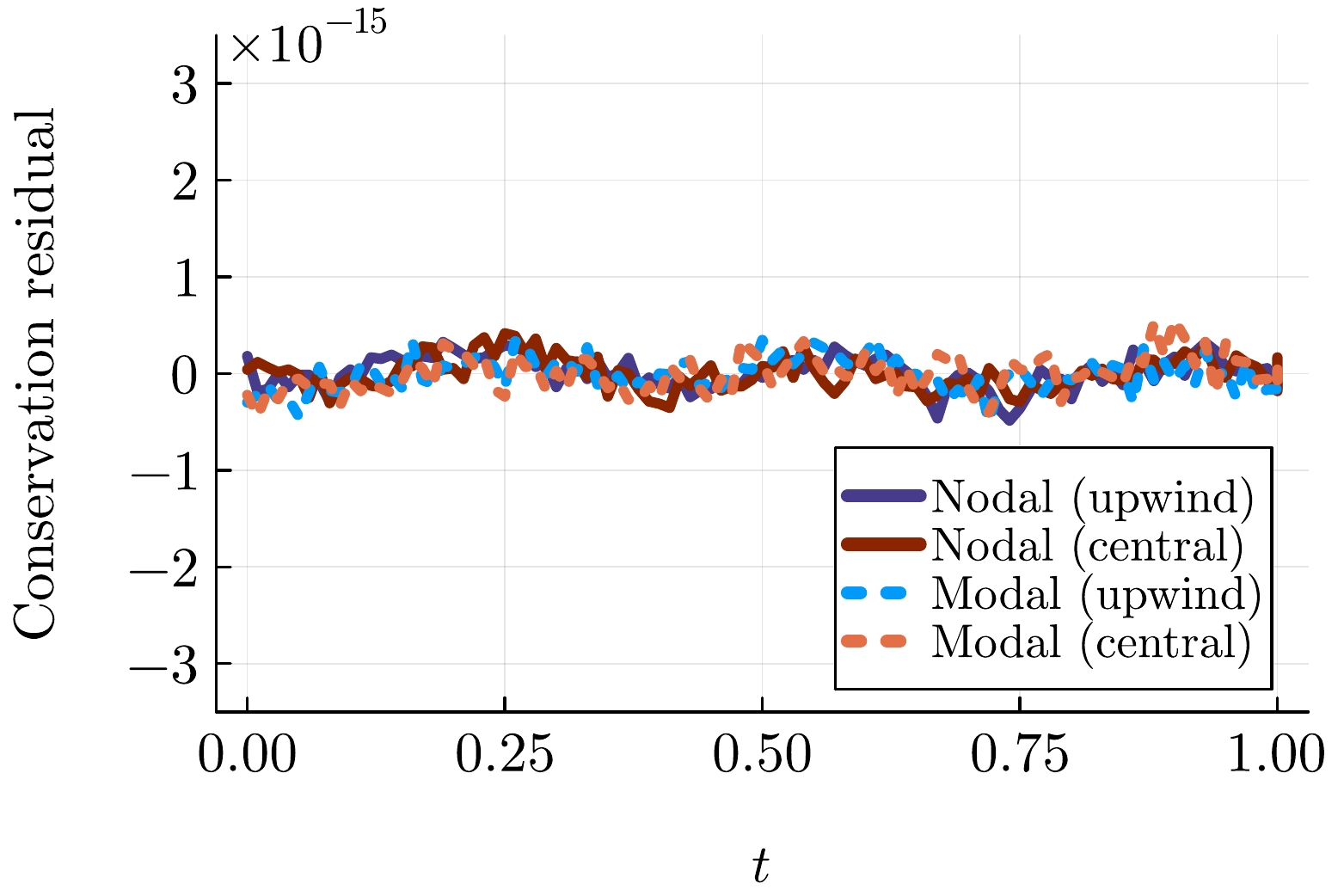}
\end{subfigure}
\begin{subfigure}{0.495\textwidth}
\centering
\includegraphics[height=36mm]{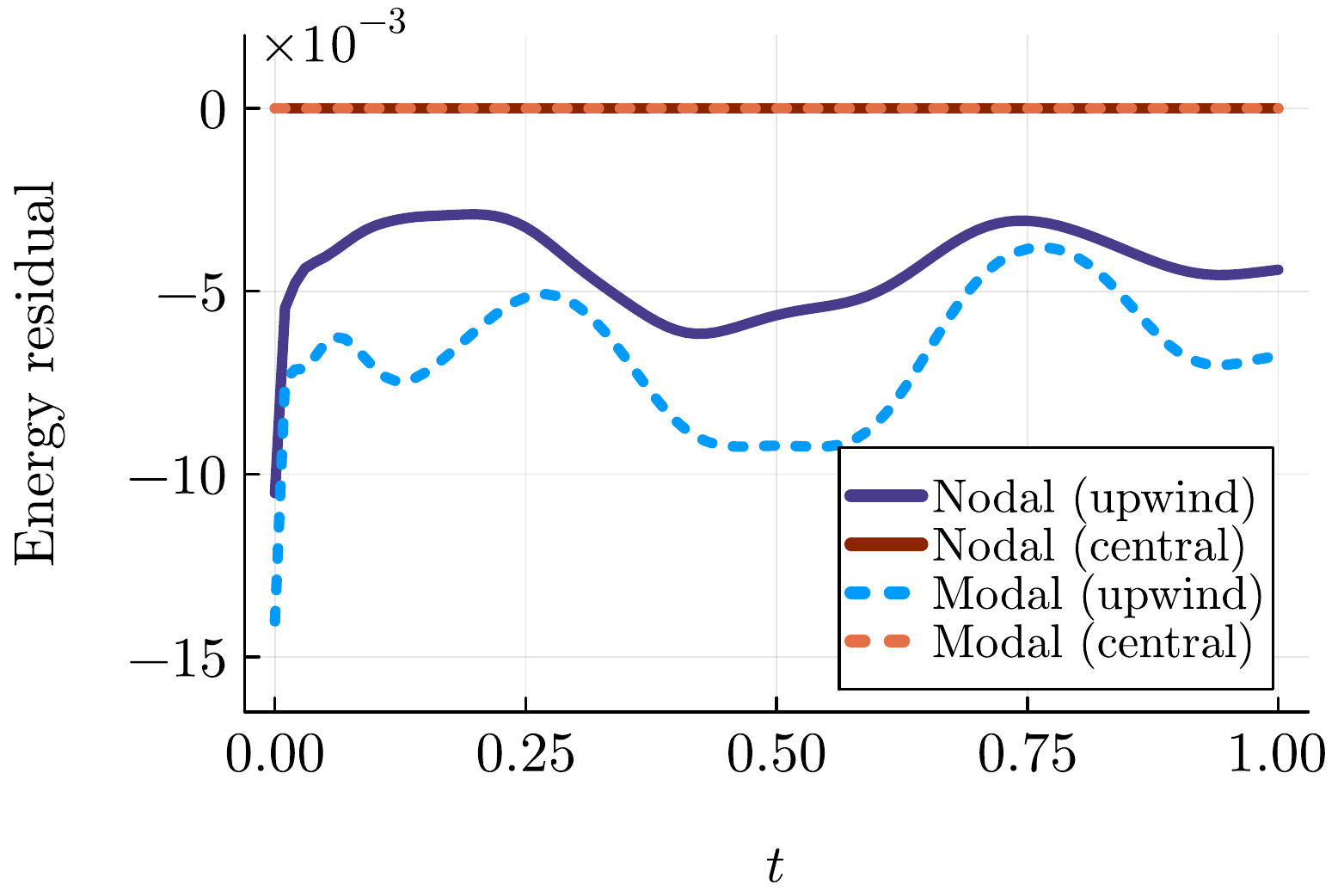}
\end{subfigure}\\
\caption{Time evolution of the conservation and energy residuals for tensor-product discretizations on triangles (top row) and tetrahedra (bottom row) plotted at 101 equispaced snapshots}\label{fig:conservation_stability}
\end{figure}
Considering nodal formulations as well as modal formulations\footnote{The weight-adjusted approximation in \cref{eq:weight_adjusted_inverse} is used for all results involving modal formulations.}\ using the proposed tensor-product SBP operators on triangles and tetrahedra, we plot the time evolution of the conservation residual and the energy residual in \cref{fig:conservation_stability}, where we present results for $p=4$ and $M=2$ as an illustrative example. Such quantities are given in the present context by
\begin{subequations}
\begin{align}
\text{Conservation residual} \, &\coloneqq \, \sum_{\kappa=1}^{\nelem}\big(\vc{1}^{(\nvolnodes)}\big)^\T\mat{W}\mat{J}^{(\kappa)}\dv{\vc{u}^{(h,\kappa)}(t)}{t}, \\
\text{Energy residual} \, &\coloneqq  \, \begin{cases} 
\begin{alignedat}{3}   
\displaystyle &\sum_{\kappa=1}^{\nelem}\big(\vc{u}^{(h,\kappa)}(t)\big)^\T\mat{W}\mat{J}^{(\kappa)}\dv{\vc{u}^{(h,\kappa)}(t)}{t}\quad &&(\text{nodal})\\
\displaystyle &\sum_{\kappa=1}^{\nelem}\big(\vc{\tilde{u}}^{(h,\kappa)}(t)\big)^\T\mat{\tilde{M}}^{(\kappa)}\dv{\vc{\tilde{u}}^{(h,\kappa)}(t)}{t}\quad &&(\text{modal})
\end{alignedat}
\end{cases},\label{eq:energy_residual}
\end{align}
\end{subequations}
corresponding to the time derivative of the discretely integrated numerical solution and that of the discrete solution energy, respectively. As expected for conservative and energy-stable SBP discretizations, the conservation residual remains on the order of machine precision for both the upwind and central variants of the nodal and modal tensor-product schemes on triangles as well as tetrahedra, while the energy residual is on the order of machine precision for a central numerical flux and negative for an upwind numerical flux. Note that we have deliberately used coarse meshes for such tests in order to demonstrate that the conservation and energy stability properties are satisfied discretely at finite resolution, rather than only in the limit of mesh refinement.

\subsection{Spectral radius}
\begin{figure}[t]
\centering
\begin{subfigure}{0.495\textwidth}
\centering
\includegraphics[height=36mm]{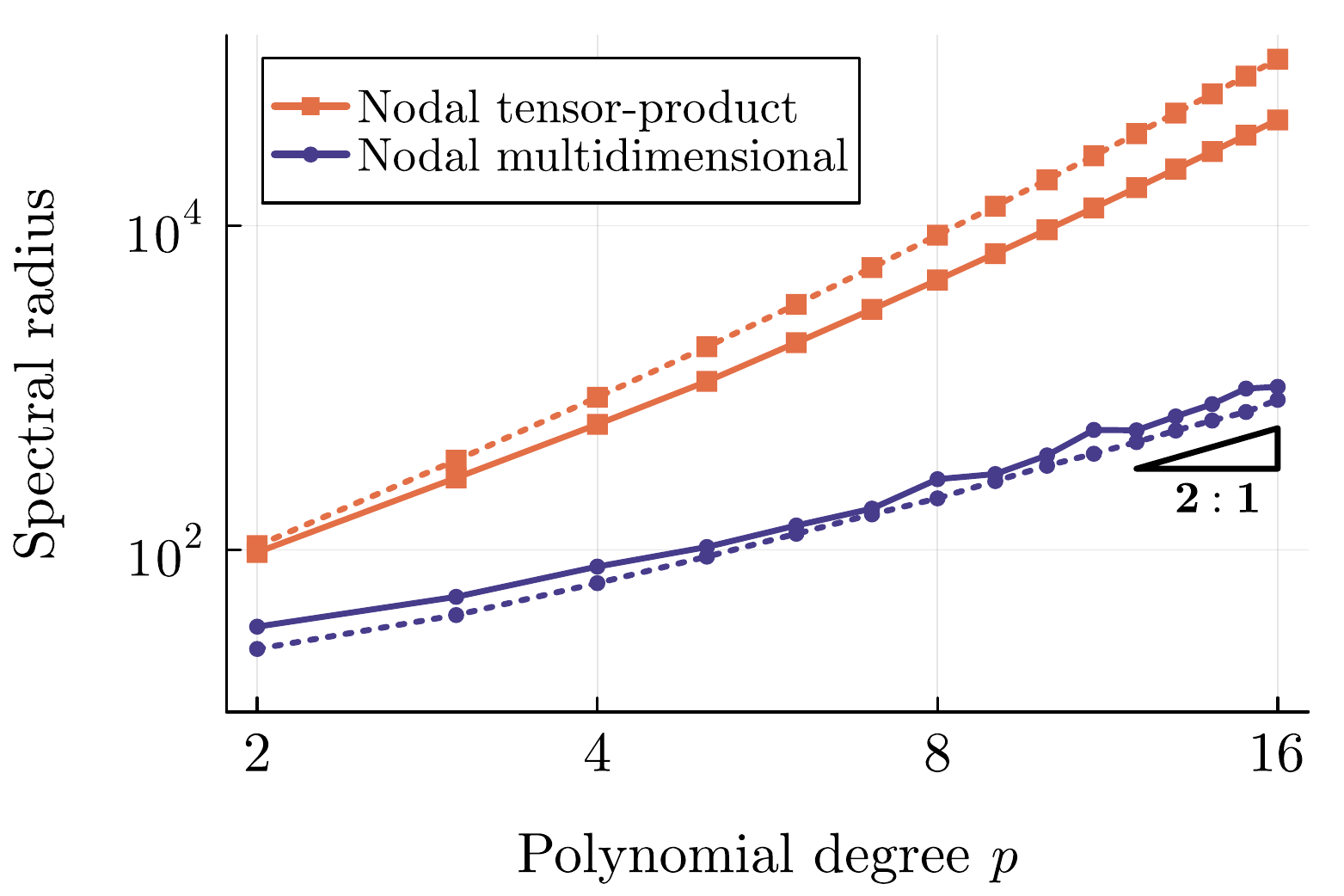} 
\end{subfigure}
\begin{subfigure}{0.495\textwidth}
\centering
\includegraphics[height=36mm]{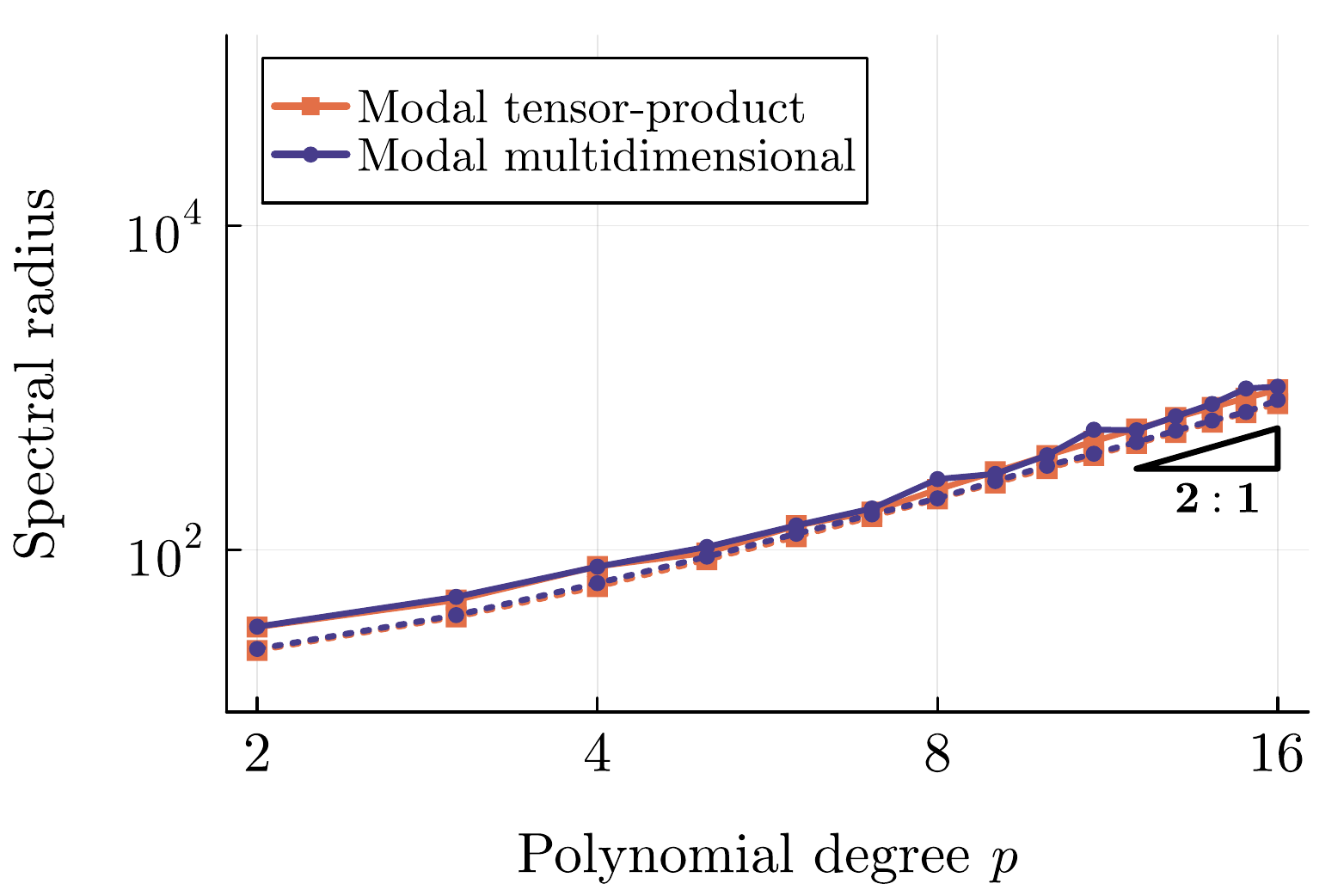}
\end{subfigure}\\
\begin{subfigure}{0.495\textwidth}
\centering
\includegraphics[height=36mm]{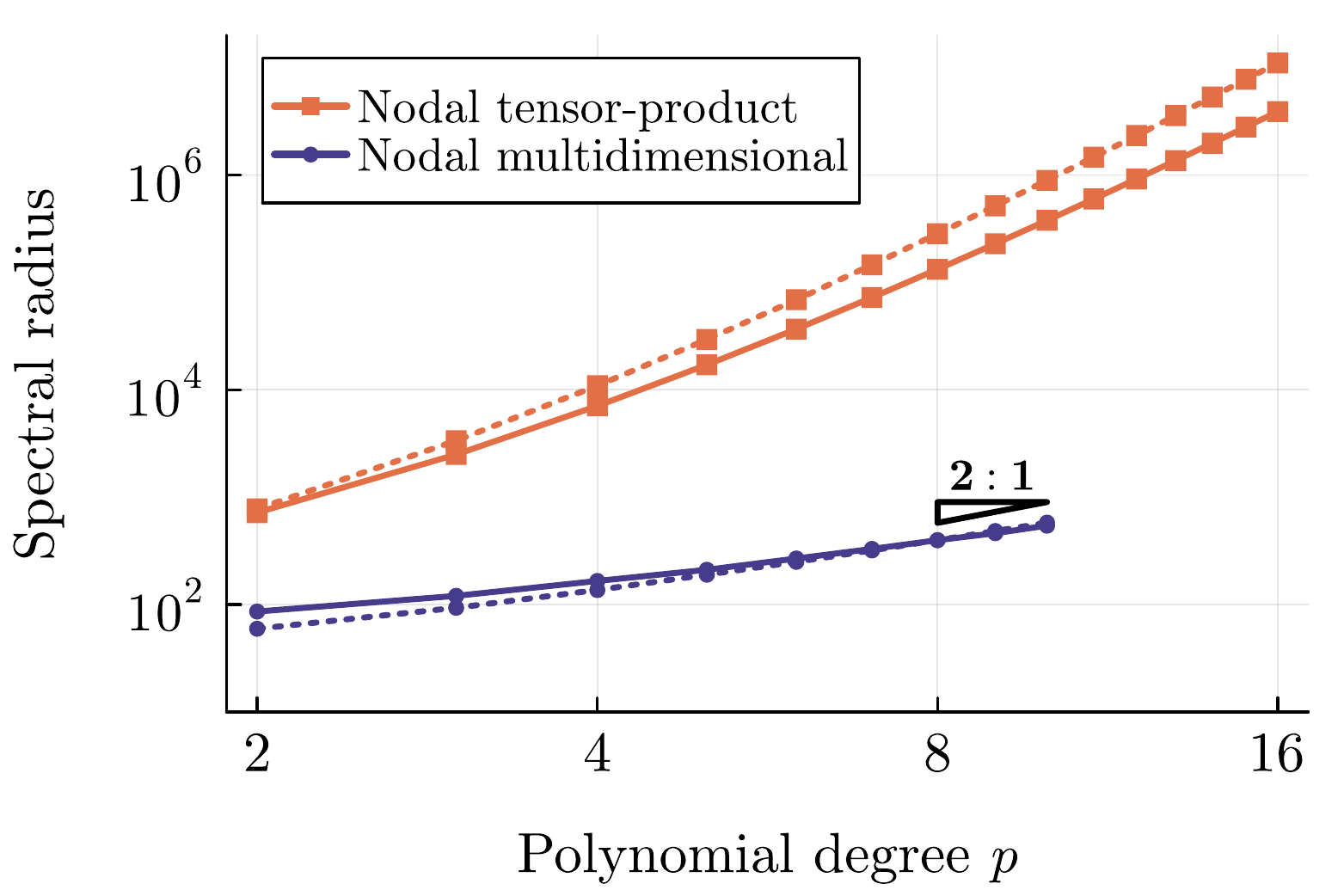}
\end{subfigure}
\begin{subfigure}{0.495\textwidth}
\centering
\includegraphics[height=36mm]{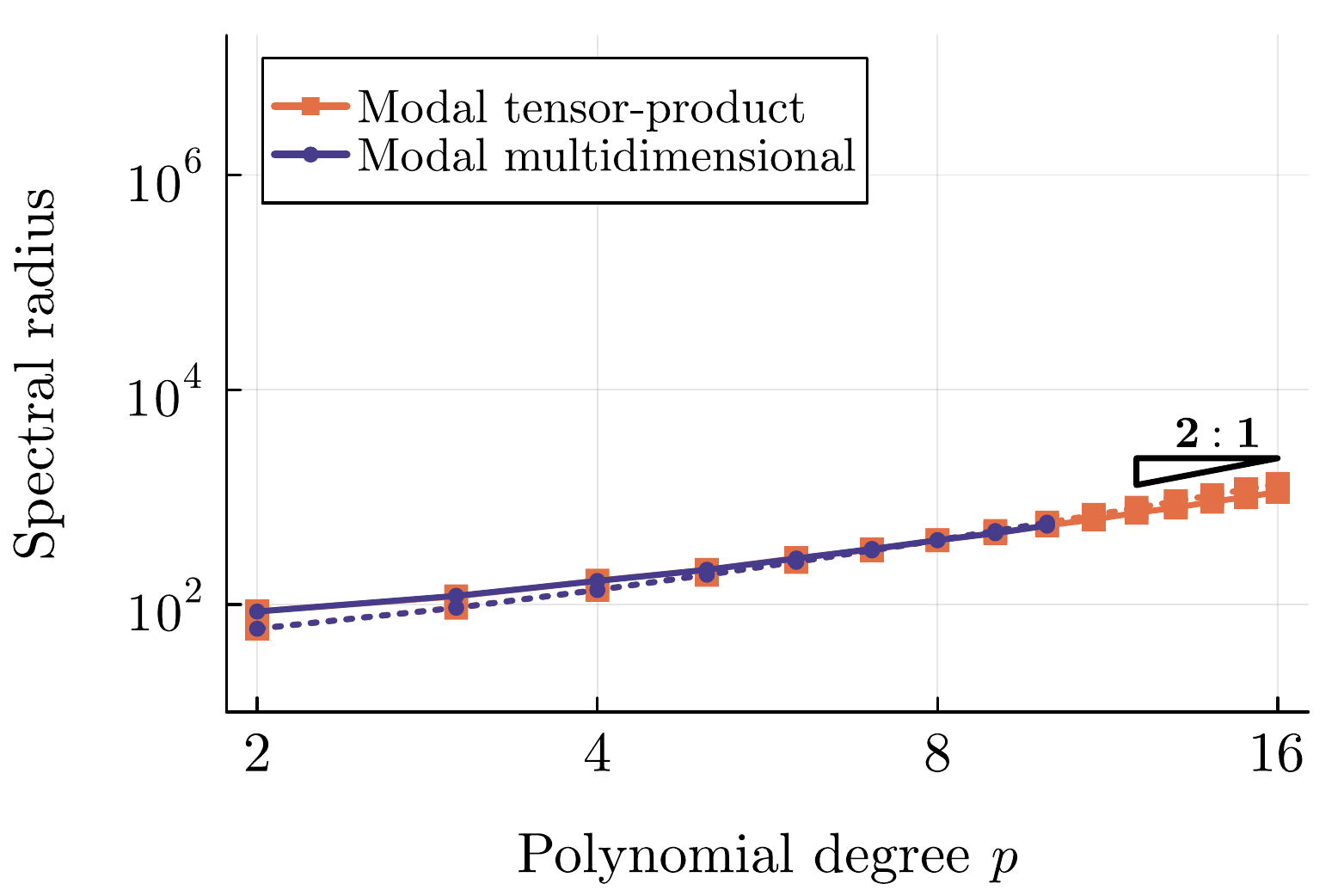}
\end{subfigure}
\caption{Variation in spectral radius of the semi-discrete advection operator with polynomial degree for discretizations on triangles (top row) and tetrahedra (bottom row); solid and dashed lines denote upwind and central numerical fluxes, respectively}\label{fig:specr}
\end{figure}
While the proposed spatial discretizations of the linear advection equation are provably energy stable in a semi-discrete sense, the stability of the fully discrete problem requires the spectrum of the semi-discrete operator to lie within the stability region of the chosen time-marching method. Thus, for explicit time integration, the maximum stable time step size is dictated by the spectral radius of the global semi-discrete operator. \Cref{fig:specr} illustrates the effect of varying the polynomial degree on the spectral radius of the semi-discrete advection operator for nodal and modal formulations using multidimensional as well as tensor-product operators on triangles and tetrahedra, keeping the mesh fixed at $M = 2$. Considering the nodal and modal multidimensional schemes, the spectral radius grows roughly quadratically with the polynomial degree in the case of a central flux, with slightly slower growth observed for an upwind flux. The spectral radii as well as their growth rates with respect to the polynomial degree are much larger for the nodal tensor-product schemes than for all other methods, which, as predicted in \cref{sec:nodal_modal_expansions}, results in a severe restriction on the time step when such schemes are used with explicit temporal integration. This is remedied through the use of the modal formulation, for which the spectral radii are very similar to those of the multidimensional schemes. Such behaviour is consistent with the literature (see, for example, \cite[section 6.3]{karniadakis_sherwin_spectral_hp_element}) and favours the use of the modal formulation at higher polynomial degrees, at least for explicit schemes.

\subsection{Accuracy}\label{sec:accuracy}
Evaluating the discrete $L^2$ error for each scheme using its associated volume quadrature rule and plotting with respect to the nominal element size $h = 1/M$, we see from \cref{fig:accuracy} that for $p = 4$, all methods considered converge as $O(h^{p+1})$ in the case of an upwind numerical flux, with similar accuracy levels on a given mesh observed for all schemes. While the specific formulations considered in this study are comparable to one another in accuracy for a given mesh, we note that the error magnitudes could be reduced while maintaining the same order of accuracy through the use of over-integration (i.e.\ increasing the number of volume quadrature nodes) for accuracy rather than stability purposes, a possibility not explored in this paper. Performing $p$-refinement with $M=2$ using a nodal formulation with multidimensional operators and a modal formulation with tensor-product operators, we see that on triangles, the modal tensor-product formulation is slightly more accurate than the nodal multidimensional formulation for high polynomial degrees, while on tetrahedra, the schemes are of comparable accuracy for degrees in which both types of operators are available.\footnote{Symmetric quadrature rules suitable for the construction of multidimensional diagonal-norm SBP operators of degree $p > 10$ on the tetrahedron are not available in the literature, to the best of the authors' knowledge. Moreover, the analysis in \cref{sec:cost} suggests that schemes employing such operators, if they were to be constructed, would be less efficient than those which we propose.}\ Such results indicate that for this model problem, the proposed tensor-product operators on triangles and tetrahedra are at least as accurate as their multidimensional counterparts for a given mesh and polynomial degree.
\begin{figure}[t!]
\centering
\begin{subfigure}{0.325\textwidth}
\includegraphics[height=36mm]{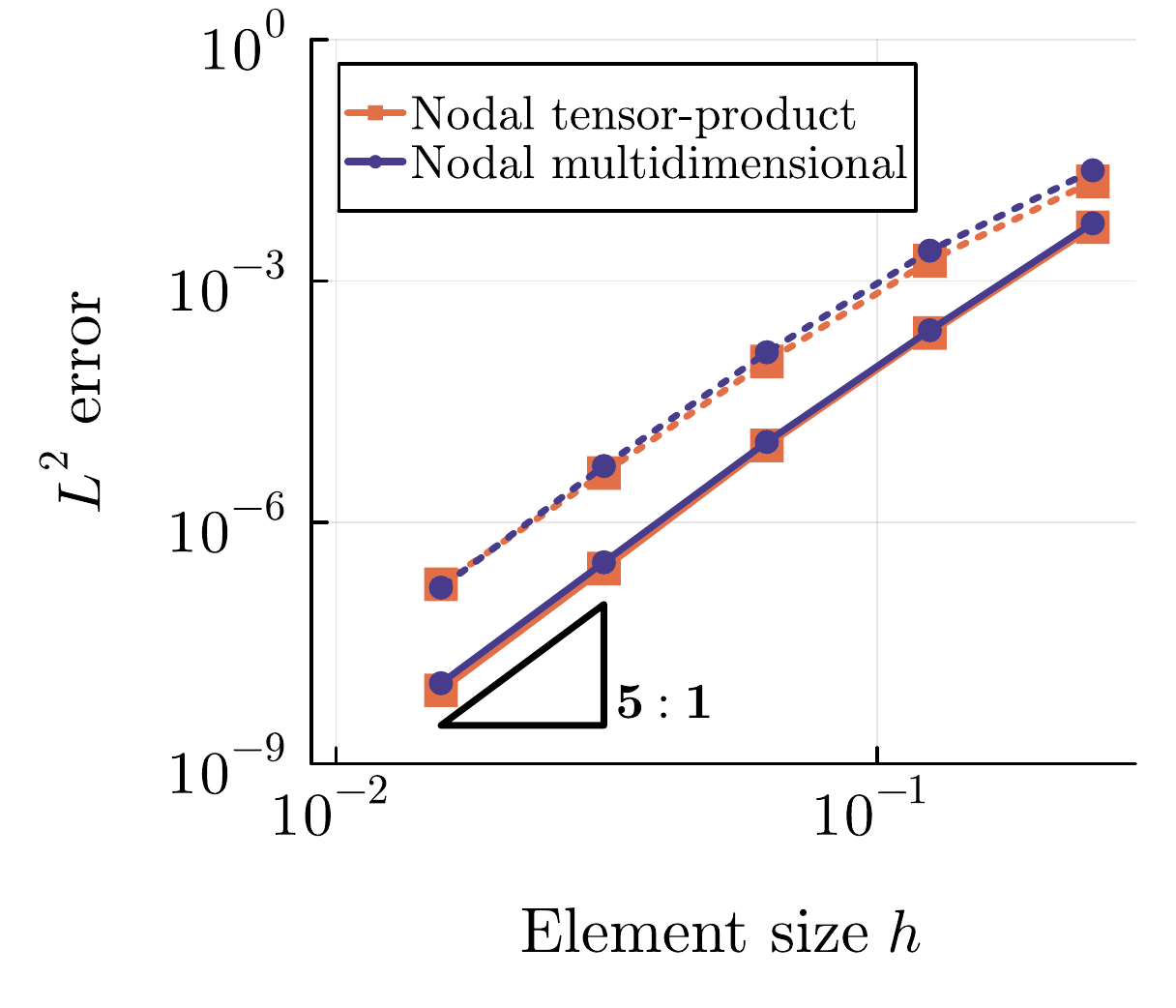}
\end{subfigure}~
\begin{subfigure}{0.325\textwidth}
\includegraphics[height=36mm]{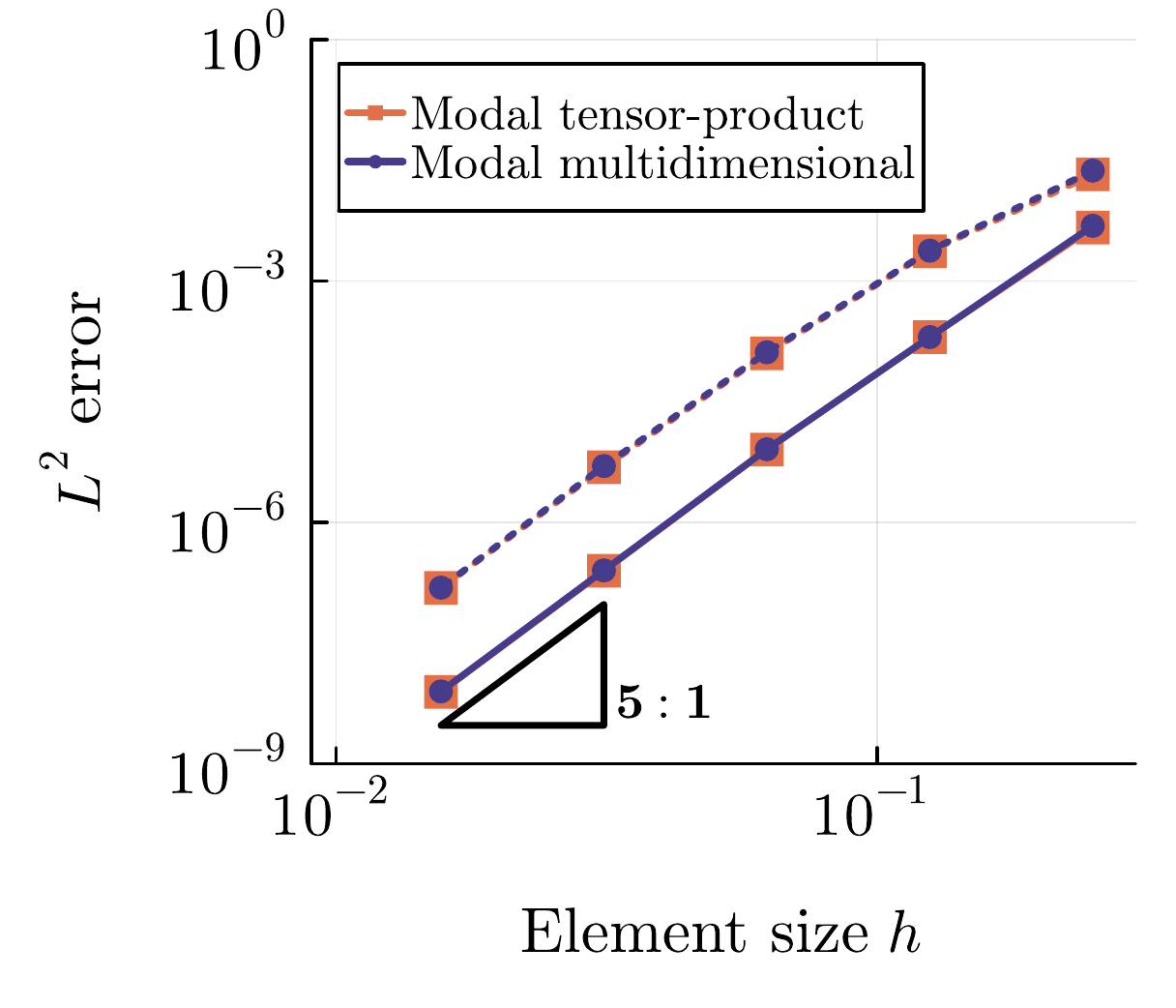}
\end{subfigure}~
\begin{subfigure}{0.325\textwidth}
\includegraphics[height=36mm]{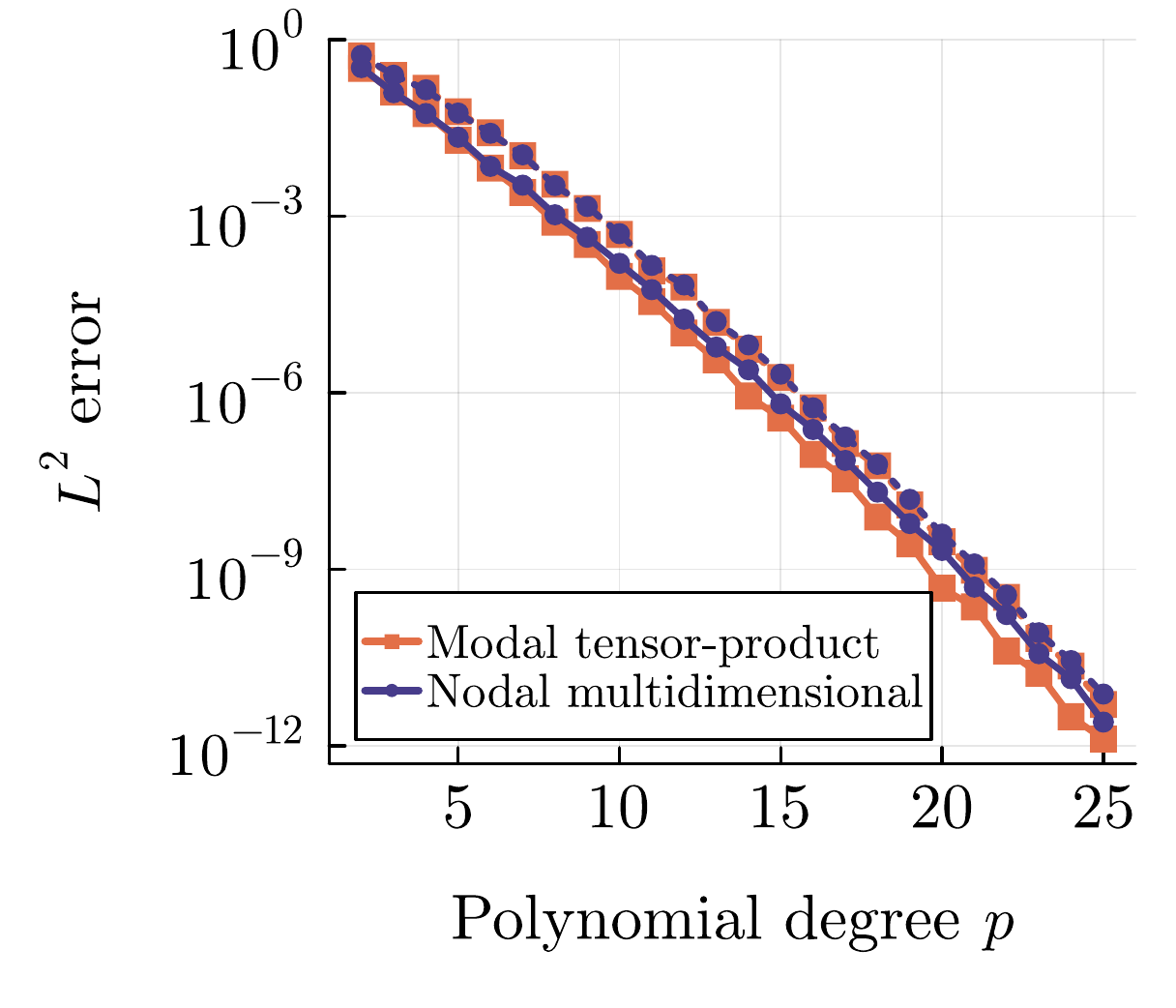}
\end{subfigure}\\
\begin{subfigure}{0.325\textwidth}
\centering
\includegraphics[height=36mm]{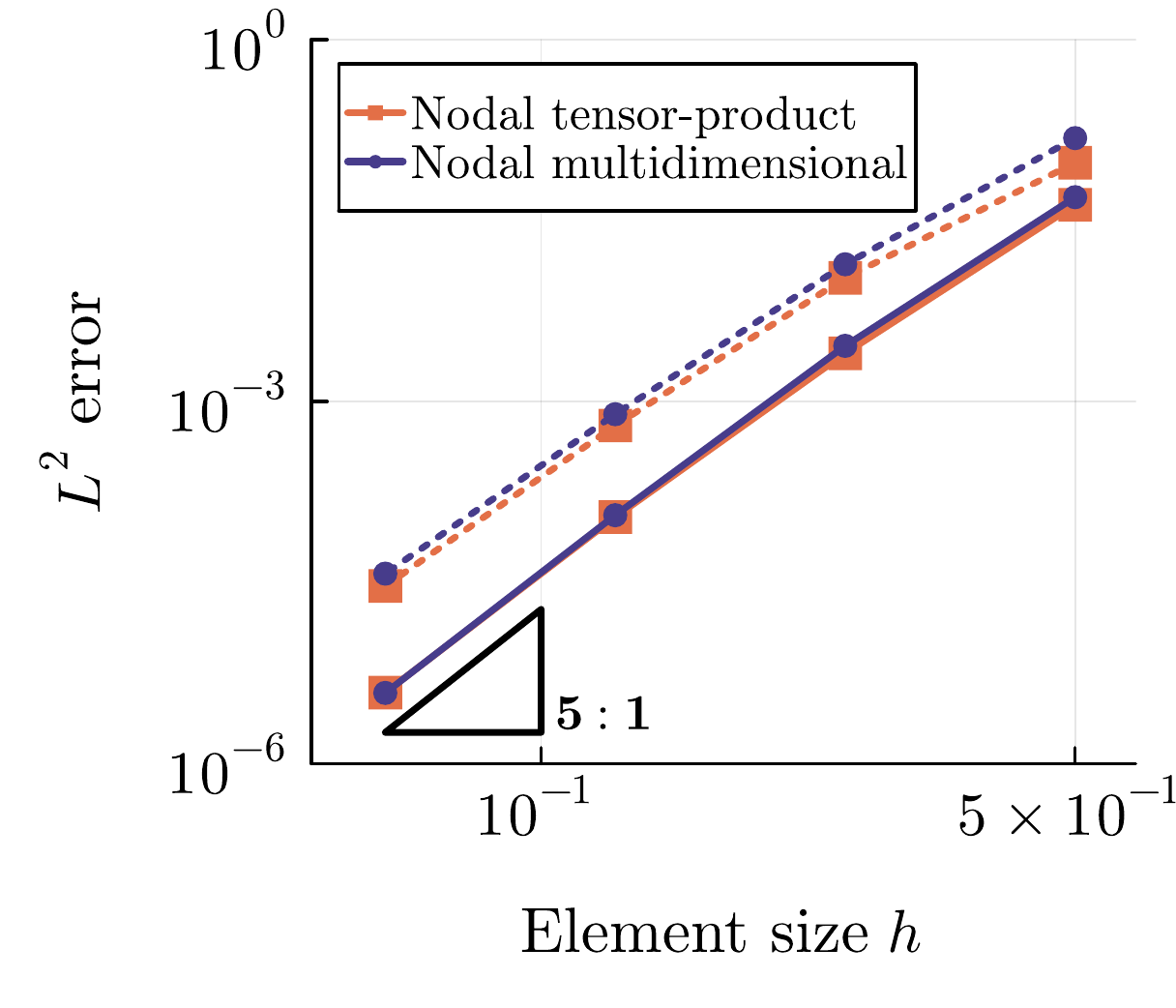}
\end{subfigure}
\begin{subfigure}{0.325\textwidth}
\centering
\includegraphics[height=36mm]{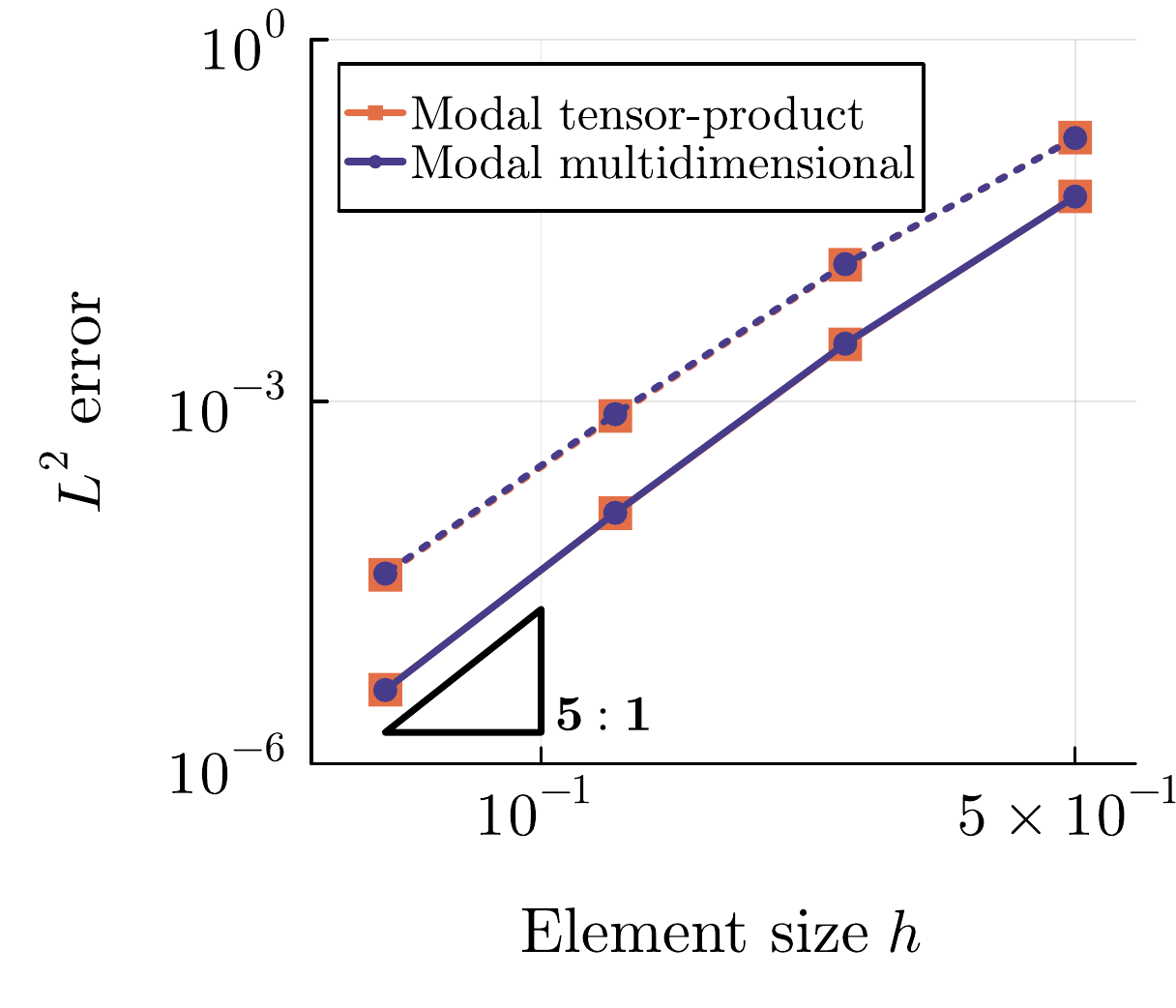}
\end{subfigure}
\begin{subfigure}{0.325\textwidth}
\centering
\includegraphics[height=36mm]{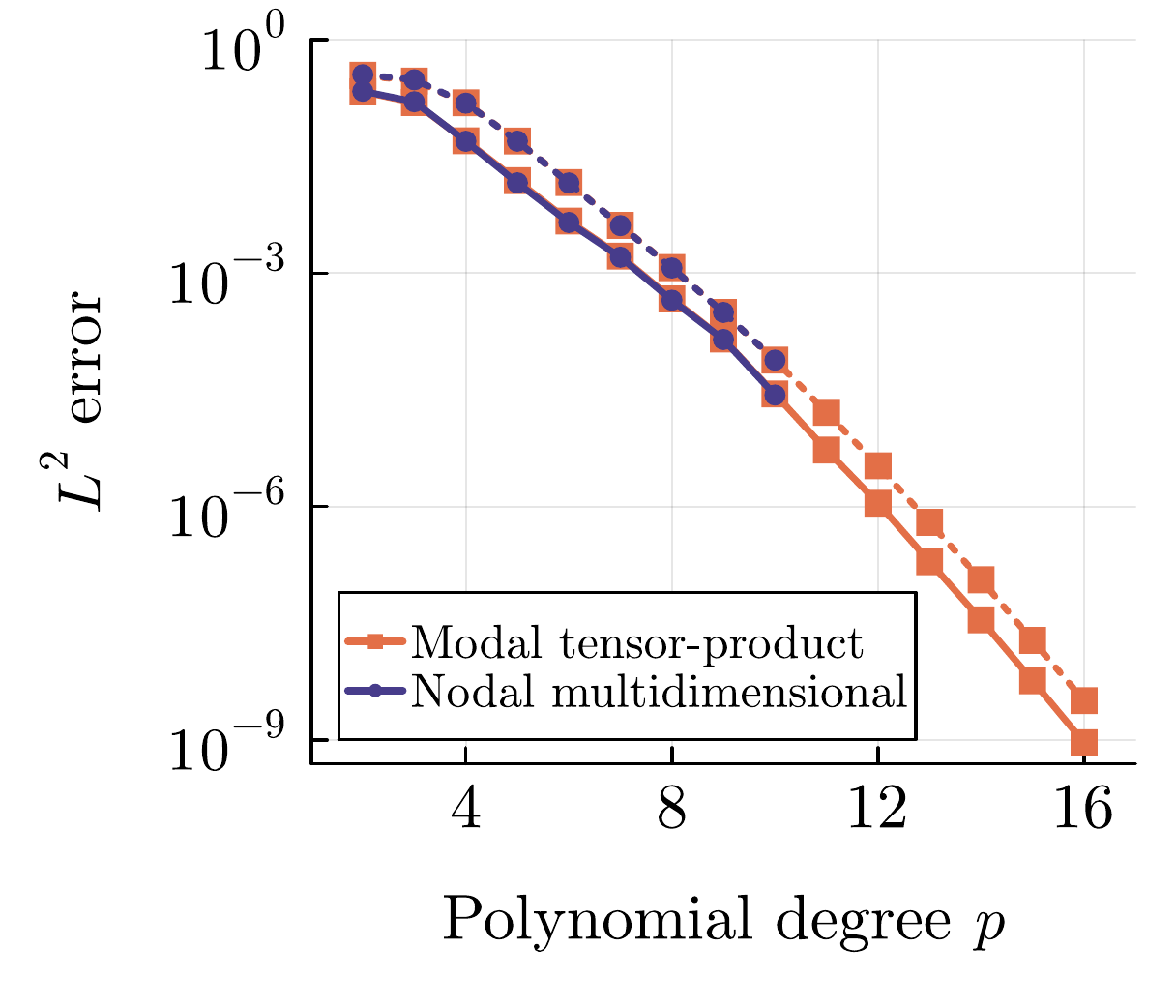}
\end{subfigure}
\caption{Convergence with respect to $h$ and $p$ for discretizations of the linear advection equation on triangles (top row) and tetrahedra (bottom row); solid and dashed lines denote upwind and central numerical fluxes, respectively}\label{fig:accuracy}
\end{figure}

\subsection{Operation count}\label{sec:cost}
\Cref{fig:cost} displays the number of floating-point operations incurred in evaluating the time derivative on a single element for each scheme at varying polynomial degrees, including the evaluation of the physical flux at all volume quadrature nodes and the evaluation of the numerical flux at all facet quadrature nodes.\footnote{The operation count for each algorithm (implemented in native Julia without calls to BLAS) was evaluated using \texttt{GFlops.jl}, which is available at \url{https://github.com/triscale-innov/GFlops.jl}.} Since the results in \cref{sec:accuracy} indicate that for a given number of elements and a given polynomial degree, the accuracy of the proposed tensor-product approach is comparable to that of a multidimensional scheme employing a symmetric quadrature rule, such an analysis is expected to provide a fair efficiency comparison, assuming that the implementations of all methods are similarly optimized and that the resulting algorithms are compute-bound rather than memory-bound. If only the reference-operator algorithms described in \cref{sec:reference_operator} are considered (i.e.\ ignoring the dashed lines in \cref{fig:cost}), we see that the sum-factorization algorithms used for the proposed tensor-product operators yield a significant reduction in operation count relative to the multidimensional operators for all orders of accuracy, scaling approximately as $O(p^{d+1})$ for both the nodal and modal tensor-product formulations, in contrast with the $O(p^{2d})$ scaling of the nodal and modal multidimensional formulations. If, however, memory is not a limiting factor, the physical-operator algorithms described in \cref{sec:physical_operator}, which require the precomputation and storage of operator matrices for each element, allow for the multidmensional operators to be competitive with the proposed tensor-product operators at lower polynomial degrees. Specifically, if we consider the dashed lines in \cref{fig:cost}, the physical-operator implementation of the multidimensional modal formulation, while asymptotically requiring $O(p^{2d})$ operations, is nevertheless cheaper than the $O(p^{d+1})$ sum-factorization implementation of the tensor-product modal formulation for $p \leq 9$ on  triangles and $p \leq 5$ on tetrahedra.
\begin{figure}[t!]
\centering
\begin{subfigure}{0.495\textwidth}
\centering
\includegraphics[height=48mm]{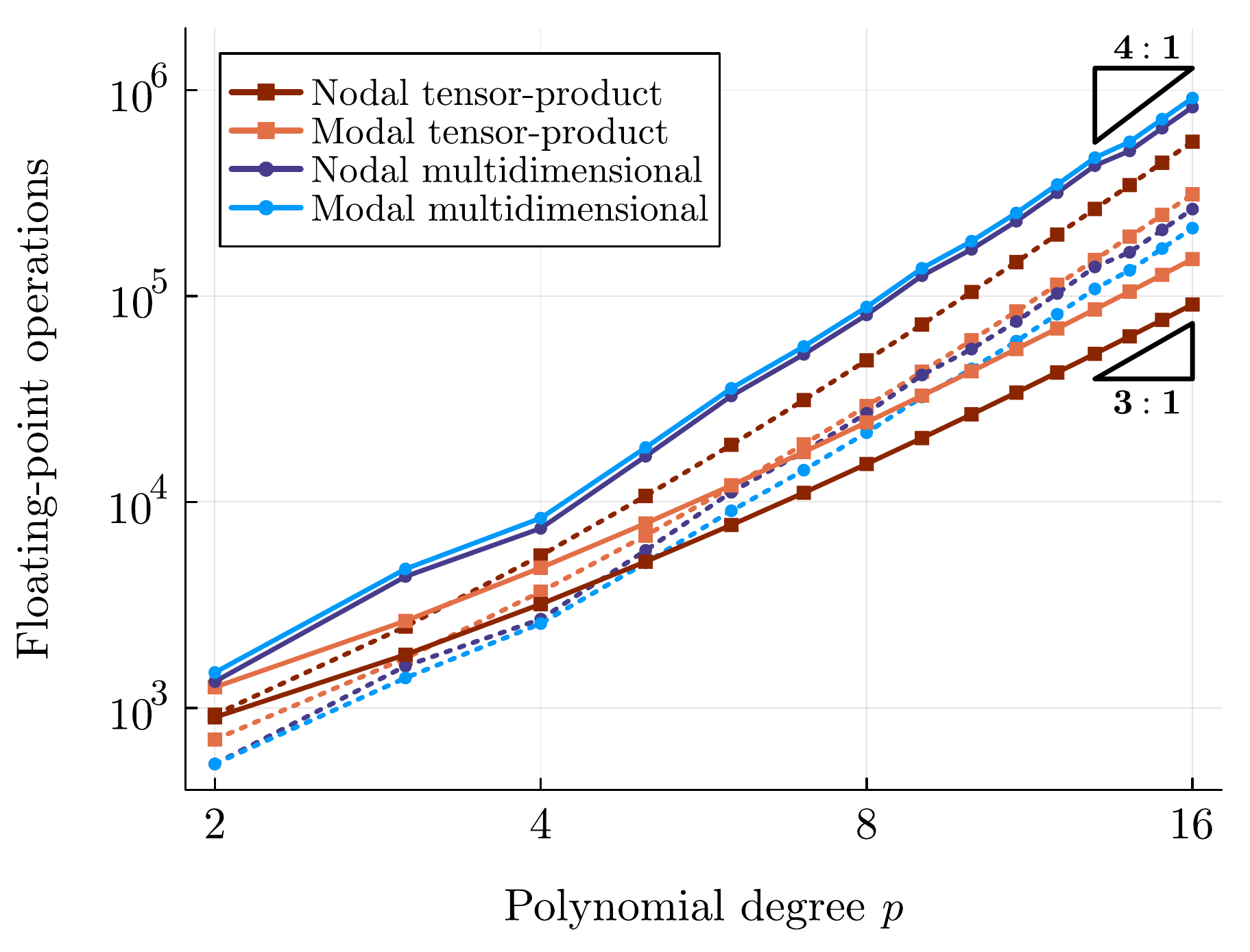}
\end{subfigure}
\begin{subfigure}{0.495\textwidth}
\centering
\includegraphics[height=48mm]{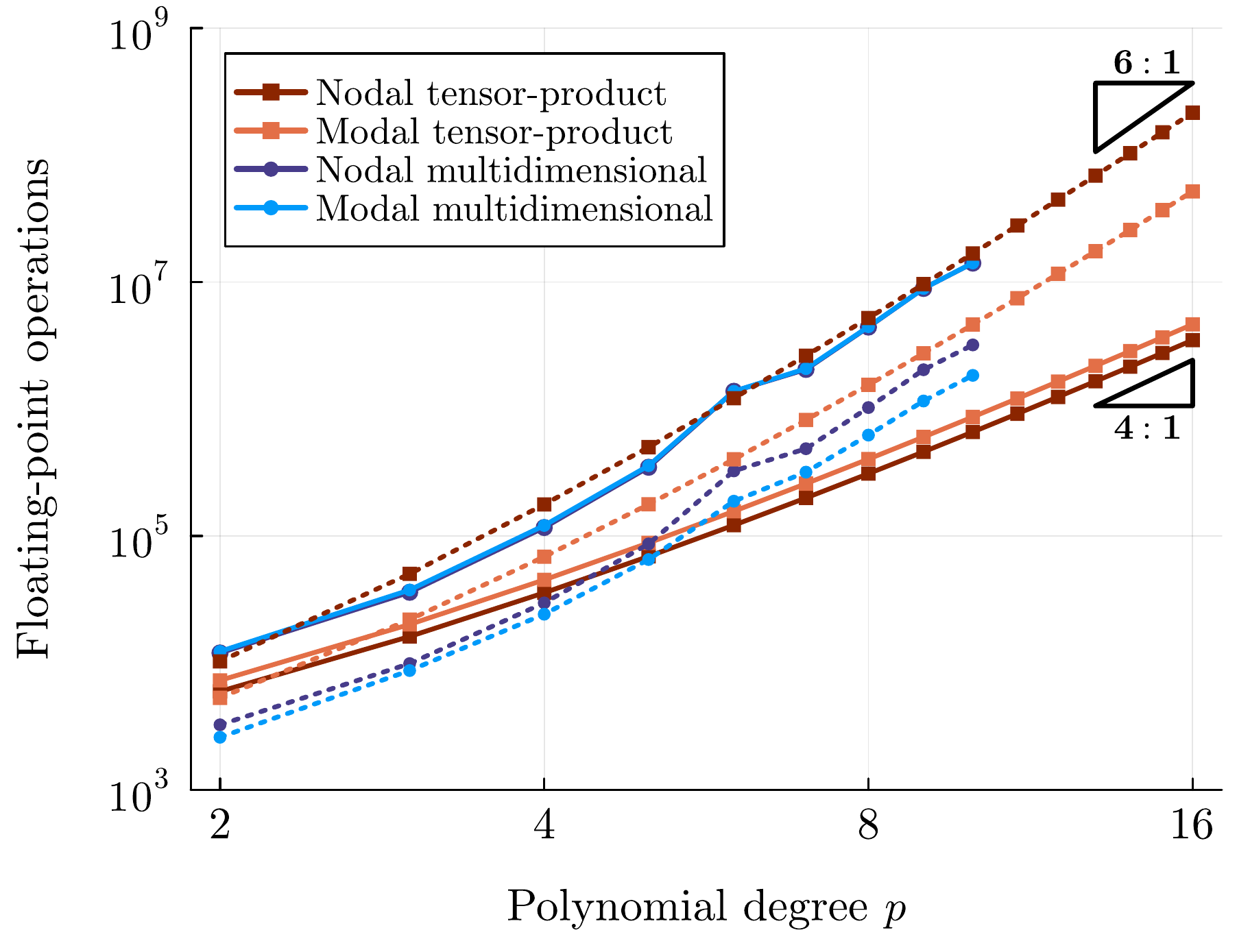}
\end{subfigure}
\caption{Floating-point operation count for local time derivative evaluation on triangles (left) and tetrahedra (right); solid and dashed lines denote reference-operator and physical-operator algorithms, respectively}\label{fig:cost}
\end{figure}

\section{Conclusions}\label{sec:conclusions}
In this paper, we have presented an extension of the SBP framework to tensor-product operators in collapsed coordinates, enabling the construction of efficient, stable, and conservative discontinuous spectral-element methods of arbitrary order on triangles and tetrahedra. Using a split formulation to obtain operators on curvilinear meshes, a projection onto the Proriol-Koornwinder-Dubiner polynomial basis to mitigate the time step restriction resulting from the singularity of the collapsed coordinate transformation, and a weight-adjusted approximation of the curvilinear modal mass matrix to explicitly obtain the local time derivative, the synergy of such technologies from a multitude of research communities has resulted in a promising methodology for the construction of efficient and robust spatial discretizations of conservation laws suitable for complex geometries. Future work includes the application of such operators within entropy-stable discretizations of nonlinear systems of conservation laws and to problems with diffusive terms, the development of suitable preconditioners for their use with implicit time integration, the construction of tensor-product SBP operators on prismatic and pyramidal elements, and efficiency comparisons involving practical problems in fluid dynamics and other disciplines.

\section*{Acknowledgments}
The authors are grateful for discussions with Jesse Chan and Gianmarco Mengaldo. Computations were performed on the Niagara supercomputer at the SciNet HPC Consortium \cite{ponce_niagara_19}, which is funded by the Canada Foundation for Innovation, the Government of Ontario, the Ontario Research Fund -- Research Excellence, and the University of Toronto.

\bibliographystyle{siamplain}
\bibliography{refs}
\end{document}